\documentclass[a4paper,11pt,reqno]{amsart}
\usepackage{latexsym}
\usepackage{color,dsfont}
\usepackage[usenames,dvipsnames]{xcolor}
\usepackage{amssymb,amsfonts,amsmath,mathrsfs}
\newtheorem{lemma}{Lemma}[section]

\newtheorem{conjecture}{Conjecture}[section]

\renewcommand{\pmod}[1]{\,(\mathrm{mod}\,#1)}

\newcommand{\acom}[1]{{\color{magenta}{\textbf{Alexandra}: #1}} }
\newcommand{\Hung}[1]{{\color{blue}{\textbf{Hung}: #1}} }
\newcommand{\kommentar}[1]{}

\begin{document}

\title[One and two level densities]{Type-I contributions to the one and two level densities of quadratic Dirichlet $L$--functions over function fields}

\author{Hung M. Bui, Alexandra Florea and J. P. Keating}
\address{Department of Mathematics, University of Manchester, Manchester M13 9PL, UK}
\email{hung.bui@manchester.ac.uk}
\address{Department of Mathematics, Columbia University, New York NY 10027, USA}
\email{aflorea@math.columbia.edu}
\address{Mathematical Institute, University of Oxford, Oxford OX2 6GG, UK}
\email{keating@maths.ox.ac.uk}

\begin{abstract}
Using the Ratios Conjecture, we write down precise formulas with lower order terms for the one and the two level densities of zeros of quadratic Dirichlet $L$--functions over function fields.  We denote the various terms arising as Type-$0$, Type-I and Type-II contributions. When the support of the Fourier transform of the test function is sufficiently restricted, we rigorously compute the Type-$0$ and Type-I terms and confirm that they match the conjectured answer.  When the restrictions on the support are relaxed, our results suggest that Type-II contributions become important in the two level density.  
\end{abstract}

\allowdisplaybreaks

\maketitle

\section{Introduction}

In this paper we compute the one and the two level densities of zeros of $L$--functions associated to quadratic characters over function fields. We compute certain Type-I contributions (as in the work of Conrey and Keating \cite{ck1, ck2, ck3, ck4, ck5}) and write down explicit conjectural Type-II terms predicted by the Ratios Conjecture \cite{cfz}. 

Understanding zeros in families of $L$--functions is a problem of considerable interest which has been much-studied. Katz and Sarnak \cite{katzsarnak, katzsarnak2} conjectured that the behavior of zeros close to the central point in a family of $L$--functions coincides with the distribution of eigenvalues near $1$ of matrices in a certain symmetry group associated to the family. There is an abundance of papers in the literature in which the above mentioned agreement is observed (for example \cite{ILS, hughesrudnick, miller, miller2}). 

When computing the $n$--level density of zeros for a particular family of $L$--functions, the Katz and Sarnak conjectures predict the main term in the asymptotic formula. Conrey, Farmer and Zirnbauer \cite{cfz} conjectured formulas for averages of ratios of $L$--functions, and using the Ratios Conjecture, one can write down an explicit formula for the $n$--level density which recovers the Katz-Sarnak main term and further include lower order terms \cite{CS}.  In the case of the Riemann zeta-function, the resulting expressions coincide with formulas obtained earlier by Bogomolny and Keating using the Hardy-Littlewood twin-prime conjecture \cite{BK1} (see also \cite{BeKe, BK4, BK5}).

A related problem is that of computing moments in families of $L$--functions. Using analogies with random matrix theory, Keating and Snaith \cite{ks2, ksnaith} conjectured asymptotic formulas with the leading order term for moments in various families. A more refined conjecture, due to Conrey, Farmer, Keating, Rubinstein and Snaith \cite{cfkrs}, and similar in nature to the Ratios Conjecture \cite{cfz}, predicts lower order terms undetected by the random matrix models. More recent work of Conrey and Keating \cite{ck1, ck2,ck3,ck4,ck5} revisits the question of evaluating shifted moments of the Riemann zeta-function from a different perspective, and recovers the lower order terms predicted in \cite{cfkrs}. Conrey and Keating used long Dirichlet polynomials rather than the approximate functional equation, and divide the terms that arise into certain Type-$0$, Type-I and Type-II contributions (depending on the number of swaps in the shifts). This builds on previous work in the case of the $n$-point correlation of the zeros by Bogomolny and Keating \cite{BK2, BK3}, where a similar division was first introduced (see also \cite{ck6, ck7}).  Here we use the same ideas to examine asymptotic formulas including lower order terms for the $n$ level density of zeros. Throughout our paper, we use the Conrey and Keating nomenclature for Type-$0$, Type-I and Type-II terms. 

For the family of quadratic Dirichlet $L$--functions, \"{O}zl\"{u}k and Snyder \cite{ozluk} computed the one level density of zeros when the support of the Fourier transform of the test function is in $(-2,2)$. The higher densities in this family of $L$--functions were studied by Rubinstein \cite{rubinstein}. For a Schwartz test function $f \in \mathcal{S}(\mathbb{R}^n)$, even in all the variables, Rubinstein computed the $n$--level density when the Fourier transform of $f$ is supported in $\sum_{j=1}^n |u_j|<1$, conditional on the Generalized Riemann Hypothesis. Gao \cite{gao} attempted to double the range in Rubinstein's result. More specifically, he showed that if $f$ is of the form $f(x_1,\ldots,x_n)= \prod_{i=1}^n f_i(x_i)$ and each $\hat{f_i}$ is supported in $|u_i|<s_i$ and $\sum_{i=1}^n s_i<2$, then the $n$--level density of zeros is equal to a complicated combinatorial factor $A(f)$. For $n=2,3$, he showed that $A(f)$ agrees with the Katz and Sarnak conjecture. Recent work of Entin, Roditty-Gershon and Rudnick \cite{entin} showed that indeed the combinatorial factor $A(f)$ obtained by Gao matches the random matrix theory prediction for all $n$. Their novel approach does not involve doing the combinatorics directly, but passing to a function field analog of the problem, taking the limit $q \to \infty$ and using equidistribution results of Katz and Sarnak.  An alternative approach was developed in \cite{CS2, MS}. 

In the function field setting, Rudnick \cite{R} computed the one level density of zeros for the family of quadratic Dirichlet $L$--functions and showed that there is a transition when the support of the Fourier transform goes beyond $1$. Bui and Florea \cite{bf} obtained infinitely many lower order terms when the support of the Fourier transform is in certain ranges, and further computed the pair correlation of zeros in the family. 

In the present paper, we consider the two level density of zeros in the family of quadratic Dirichlet $L$--functions. Let $\mathcal{H}_{2g+1}$ denote the space of monic, square-free polynomials of degree $2g+1$ over $\mathbb{F}_q[x]$. For simplicity, in the definition of the two level density, we take the test function to be equal to $1$. The two level density of zeros is defined to be
\begin{equation}
I_2(N;\alpha,\beta)=\frac{1}{|\mathcal{H}_{2g+1}|}\sum_{D\in\mathcal{H}_{2g+1}}\sum_{\substack{f_1,f_2\in\mathcal{M}\\d(f_1f_2)\leq N}}\frac{\Lambda(f_1)\Lambda(f_2)\chi_D(f_1f_2)}{|f|_1^{1/2+\alpha}|f|_2^{1/2+\beta}}, \label{i2}
\end{equation}
where $\Lambda(f)$ denotes the von Mangoldt function over function fields, and $\chi_D(f)$ is the quadratic character.

 Using the Ratios Conjecture over function fields \cite{AK}, we write down precise formulas for the two level density in terms of Type-$0$, Type-I and Type-II contributions. The Type-I terms kick in when $N \geq 2g$ and Type-II terms appear when $N \geq 4g$. We  compute the Type-$0$ and Type-I terms rigorously by estimating sums over primes (i.e.~over monic irreducible polynomials). Our approach in computing the two level density is more direct than the one used by Entin, Roditty-Gershon and Rudnick \cite{entin}, and we do not take $q \to \infty$ (hence we do not use any equidistribution results). The Type-$0$ terms, or the so-called "diagonal", come from prime powers $f_1$ and $f_2$ in \eqref{i2} with the product $f_1 f_2$ being a square. The diagonal terms are relatively straightforward to compute. Evaluating the Type-I terms is more subtle and requires more involved computations. We use the Poisson summation formula for the sum over $D$ (after removing the squarefree condition) and then we compute the contribution from the parameter on the dual side of the Poisson summation formula being a square. We sum up these contributions and then we check that they match the answer conjectured from the Ratios Conjecture. 
 
 Type-I terms essentially come from squares on the dual side of the Poisson summation formula over function fields. Our methods do not allow us to identify the Type-II terms which only arise when $N \geq 4g$, but we explicitly write down the conjectured Type-II contribution.  This is one of our main goals: to draw attention to the fact that when the methods that have been employed successfully for many years in calculations of the one level density are applied to the two level density they fail to capture all of the terms, underlining the importance of developing methods to compute the Type-II terms in this case.
 
For the sake of completeness, we also include the computation of the one level density (with a shift) and match the terms we obtain with the Type-$0$ and Type-I contributions. 
 
 \subsection{Outline of the paper} 
In Section \ref{background} we gather a few useful lemmas we will need. In Section \ref{1rc} we use the Ratios Conjecture to write down formulas for the one level density of zeros with Type-$0$ and Type-I terms (there are no Type-II terms for the one level density). We rigorously compute these terms when $N<4g$ and match them to the conjecture in Section \ref{1compute}. In Section \ref{2rc} we again use the Ratios Conjecture to predict the Type-$0$, Type-I and Type-II contributions for the two level density. The diagonal terms are computed in Section \ref{2diag} and Type-I terms in Section \ref{type1}. In subsection \ref{combine} we combine the various contributions from Sections \ref{type11} and \ref{type12} and show that they agree with the conjecture.

  \medskip
{\bf Acknowledgements.} A. Florea gratefully acknowledges the support of an NSF Postdoctoral Fellowship during part of the research which led to this paper. 
J.P. Keating was supported by a Royal Society Wolfson Research Merit Award, EPSRC Programme Grant EP/K034383/1 LMF: $L$-Functions and Modular Forms, and by ERC Advanced Grant 740900 (LogCorRM).
 The authors would also like to thank Julio Andrade, Brian Conrey, Chantal David, Steve Gonek and Matilde Lal\'{i}n for many stimulating discussions and useful comments during SQuaRE meetings at AIM.
 
\section{Lemmas}
\label{background}
Let $q \equiv 1 \pmod 4$ be a prime. We denote the set of monic polynomials over $\mathbb{F}_q[x]$ by $\mathcal{M}$. Let $\mathcal{M}_n$ denote the set of monic polynomials of degree $n$, $\mathcal{H}_n$ the set of monic, squarefree polynomials of degree $n$, and $\mathcal{P}_n$ the monic, irreducible polynomials of degree $n$. The set of monic polynomials of degree less than or equal to $n$ is denoted by $\mathcal{M}_{\leq n}$. For simplicity, we denote the degree of a polynomial $f$ by $d(f)$. 
The norm of a polynomial $f$ is defined by $|f|= q^{d(f)}$. 

 The zeta-function over $\mathbb{F}_q[x]$ is defined by 
$$\zeta_q(s) = \sum_{f\in\mathcal{M}} \frac{1}{|f|^s}$$ for $\Re(s)>1$. Since there are $q^n$ monic polynomials of degree $n$, one can easily show that
$$\zeta_q(s) = \frac{1}{1-q^{1-s}},$$ and this provides a meromorphic continuation of $\zeta_q$ with a simple pole at $s=1$. Making the change of variables $u=q^{-s}$, the zeta-function becomes
$$ \mathcal{Z}(u) = \zeta_q(s) = \sum_{f \in\mathcal{M}} u^{d(f)} = \frac{1}{1-qu},$$ which has a simple pole at $u=1/q$. Note that $\mathcal{Z}(u)$ is given by the Euler product
$$ \mathcal{Z}(u) = \prod_P \Big(1-u^{d(P)}\Big)^{-1},$$ for $|u|<1/q$, where the product is over monic, irreducible polynomials in $\mathbb{F}_q[t]$.

The quadratic character over $\mathbb{F}_q[t]$ is defined as follows. For $P$ a monic, irreducible polynomial let
$$ \Big( \frac{f}{P} \Big)= 
\begin{cases}
1 & \mbox{ if } P \nmid f, f \text{ is a square modulo }P, \\
-1 & \mbox{ if } P \nmid f, f \text{ is not a square modulo }P, \\
0 & \mbox{ if } P|f.
\end{cases}
$$
We extend the definition of the quadratic residue symbol above to any monic $D \in \mathbb{F}_q[t]$ by multiplicativity, and define the quadratic character $\chi_D$ by
$$\chi_D(f) = \Big( \frac{D}{f} \Big).$$
Since we assumed that $q \equiv 1 \pmod 4$, note that the quadratic reciprocity law takes the following form: if $A$ and $B$ are two monic coprime polynomials, then 
$$ \Big( \frac{A}{B} \Big) = \Big( \frac{B}{A} \Big).$$
We define the von Mangoldt function to be
$$\Lambda(f) = 
\begin{cases}
d(P) & \mbox{ if } f=cP^k, c \in \mathbb{F}_q^{\times}, \\
0 & \mbox{ otherwise.}
\end{cases}
$$
The following lemma expresses sums over squarefree polynomials in terms of sums over monics.

\begin{lemma}\label{L1}
For $f\in\mathcal{M}$ we have
\[
\sum_{D\in\mathcal{H}_{2g+1}}\chi_D(f)=\sum_{C|f^\infty}\sum_{h\in\mathcal{M}_{2g+1-2d(C)}}\chi_f(h)-q\sum_{C|f^\infty}\sum_{h\in\mathcal{M}_{2g-1-2d(C)}}\chi_f(h),
\]
where the summations over $C$ are over monic polynomials $C$ whose prime factors are among the prime factors of $f$.
\end{lemma}
\begin{proof}
See Lemma $2.2$ in \cite{aflorea}. 
\end{proof}

We define the generalized Gauss sum as follows. For $f \in \mathcal{M}$, let
\[
G(V,f):= \sum_{u \pmod f} \chi_f(u)e\Big(\frac{uV}{f}\Big),
\]
where the exponential over function fields was defined in \cite{hayes}. Specifically, for $a \in \mathbb{F}_q((1/t))$, 
$$e(a) = e^{2 \pi i a_1/q},$$ where $a= \ldots +a_1/t+ \ldots$.

The following two lemmas are Proposition 3.1 and Lemma 3.2 in \cite{aflorea}.

\begin{lemma}\label{L2}
Let $f\in\mathcal{M}_n$. If $n$ is even then
\[
\sum_{h\in\mathcal{M}_m}\chi_f(h)=\frac{q^m}{|f|}\bigg(G(0,f)+q\sum_{V\in\mathcal{M}_{\leq n-m-2}}G(V,f)-\sum_{V\in\mathcal{M}_{\leq n-m-1}}G(V,f)\bigg),
\]
otherwise
\[
\sum_{h\in\mathcal{M}_m}\chi_f(h)= \frac{q^{m+1/2}} {|f|}\sum_{V\in\mathcal{M}_{n-m-1}}G(V,f).
\]
\end{lemma}

\begin{lemma}\label{L3}
\begin{enumerate}
\item If $(f,h)=1$, then $G(V, fh)= G(V, f) G(V,h)$.
\item Write $V= V_1 P^{\alpha}$ where $P \nmid V_1$.
Then 
 $$G(V , P^j)= 
\begin{cases}
0 & \mbox{if } j \leq \alpha \text{ and } j \text{ odd,} \\
\varphi(P^j) & \mbox{if }  j \leq \alpha \text{ and } j \text{ even,} \\
-|P|^{j-1} & \mbox{if }  j= \alpha+1 \text{ and } j \text{ even,} \\
\chi_P(V_1) |P|^{j-1/2} & \mbox{if } j = \alpha+1 \text{ and } j \text{ odd, } \\
0 & \mbox{if } j \geq 2+ \alpha .
\end{cases}$$ 
\end{enumerate}
\end{lemma}

The following lemmas are the equivalent of the Polya-Vinogradov inequality and the Weil bound in function fields.
\begin{lemma} 
\label{pv}
We have
$$ \sum_{D \in \mathcal{H}_{2g+1}} \chi_D(P) \ll |P|^{1/2},$$
and for $Q$ a prime polynomial,
$$ \sum_{\substack{D \in \mathcal{H}_{2g+1} \\ (D,Q)=1}} \chi_D(P) \ll \frac{g}{d(Q)}|P|^{1/2} .$$
\end{lemma}
\begin{proof}
See, for example, Lemma $3.5$ and p. $8033$ in \cite{bf}.
\end{proof}

\begin{lemma}[The Weil bound]\label{sumprimes}
For $V\in\mathcal{M}$ not a perfect square we have
$$\sum_{P \in \mathcal{P}_n} \chi_V(P) \ll \frac{d(V)}{n} q^{n/2}.$$
\end{lemma}
\begin{proof}
See equation $2.5$ in \cite{R}. 
\end{proof}

\begin{lemma}\label{L4}
For $f\in\mathcal{M}$ we have
\[
\frac{1}{| \mathcal{H}_{2g+1}|}\sum_{D \in \mathcal{H}_{2g+1}} \chi_D(f^{2})=\prod_{P|f}\bigg(1-\frac{1}{|P|+1}\bigg)+O(q^{-2g}).
\]
\end{lemma}
\begin{proof}
See, for example, Lemma $3.7$ in \cite{bf}.
\end{proof}

\section{The one level density - using the Ratios Conjecture}
\label{1rc}
Consider
\begin{equation}\label{maineq}
I_1(N;\alpha)=\frac{1}{|\mathcal{H}_{2g+1}|}\sum_{D\in\mathcal{H}_{2g+1}}\sum_{f\in \mathcal{M}_{\leq N}}\frac{\Lambda(f)\chi_D(f)}{|f|^{1/2+\alpha}},
\end{equation}
where the shift is assumed to satisfy $|\alpha| \ll 1/g$.

Using an analogue of the Perron formula in the form
\begin{equation}\label{Perron}
\sum_{n\leq N}a(n)=\frac{1}{2\pi i}\oint_{|u|=r}\bigg(\sum_{n=0}^{\infty}a(n)u^{n}\bigg)\frac{du}{u^{N+1}(1-u)}
\end{equation}
we get
\begin{align*}
I_1(N;\alpha)&=\frac{1}{|\mathcal{H}_{2g+1}|}\sum_{D\in\mathcal{H}_{2g+1}}\frac{1}{2\pi i}\oint_{|u|=r}\sum_{f\in \mathcal{M}}\frac{\Lambda(f)\chi_D(f)u^{d(f)}}{|f|^{1/2+\alpha}}\frac{du}{u^{N+1}(1-u)}\\
&=\frac{1}{|\mathcal{H}_{2g+1}|}\sum_{D\in\mathcal{H}_{2g+1}} \frac{1}{2 \pi i} \oint_{|u|=r}  \frac{u}{q^{1/2+\alpha}} \frac{ \mathcal{L}'}{\mathcal{L}} \Big(\frac{u}{q^{1/2+\alpha}},\chi_D \Big) \frac{du}{u^{N+1}(1-u)}
\end{align*}
for any $r<q^{-1/2-\varepsilon}$. We enlarge the contour to $|u|=r=q^{-\varepsilon}$. The Ratios Conjecture implies that (see, for example, Theorem $8.1$ in \cite{bf})
\begin{align*}
\frac{1}{|\mathcal{H}_{2g+1}|} &\sum_{D\in\mathcal{H}_{2g+1}}u \frac{\mathcal{L}'}{\mathcal{L}} (u,\chi_D) = u^2 \frac{\mathcal{Z}'}{\mathcal{Z}}(u^2) - \mathcal{B}(u) + (qu^{2})^{g} \mathcal{A}_1(u)\mathcal{Z}\Big(\frac{1}{q^{2}u^{2}}\Big)+O_\varepsilon(q^{-g+\varepsilon g}),
\end{align*}
where \begin{equation}\label{mathcalBr}\mathcal{B}(u) = \sum_P \frac{ d(P) u^{2d(P)}}{(1-u^{2 d(P)})(|P|+1)}\end{equation} and
\begin{align*}
\mathcal{A}_1(u) &= \prod_P \Big(1- \frac{1}{|P|} \Big)^{-1} \Big( 1-\frac{1}{|P|^2u^{2 d(P)} (|P|+1)}-\frac{1}{|P|+1}\Big) \\
&= \prod_P \Big(1-\frac{1}{|P|^2} \Big)^{-1} \Big(1- \frac{1}{|P|^3 u^{2d(P)}} \Big) = \frac{\mathcal{Z}(1/q^{2})}{\mathcal{Z}(1/q^{3}u^{2})}=1+\frac{1-(qu^2)^{-1}}{q-1}.
\end{align*}
Hence, up to an error of size $O_\varepsilon(q^{-g+\varepsilon g})$,
\begin{align}\label{1}
I_1(N;\alpha)&=\frac{1}{2\pi i}\oint_{|u|=r}\frac{du}{u^{N-1}(1-u)(q^{2\alpha}-u^2)}-\frac{1}{2\pi i}\oint_{|u|=r}\frac{\mathcal{B}(u,\alpha)du}{u^{N+1}(1-u)}\\
&\qquad\quad+\frac{q^{-2g\alpha}}{2\pi i}\oint_{|u|=r}\frac{du}{u^{N-2g-1}(1-u)(u^2-q^{2\alpha})}+\frac{q^{-2g\alpha}}{2\pi i(q-1)}\oint_{|u|=r}\frac{du}{u^{N-2g+1}(1-u)},\nonumber
\end{align}
where
\begin{equation*}
\mathcal{B}(u,\alpha)=\mathcal{B}\Big(\frac{u}{q^{1/2+\alpha}}\Big).
\end{equation*}

Enlarging the contours we cross the poles at $u=1$ and $u=\pm q^{\alpha}$ in the first integral, and the only pole at $u=1$ in the second integral. Note that $\mathcal{B}(u,\alpha)$ is absolutely convergent for $|u|<q^{1/2-\varepsilon}$, so in the second integral we shift the contour to $|u|=q^{1/2-\varepsilon}$, obtaining an error term of size $O_\varepsilon(q^{-N/2+\varepsilon N})$. Hence the contribution of the first two terms in \eqref{1} is equal to
\begin{align}\label{type01}
&\frac{q^{-2[N/2]\alpha}-1}{1-q^{2\alpha}}-B(\alpha) +O_\varepsilon(q^{-N/2+\varepsilon N}),
\end{align}
where
\begin{equation}\label{Br}
B(\alpha):=\mathcal{B}(1,\alpha)= \sum_P \frac{ d(P) }{(|P|^{1+2\alpha}-1)(|P|+1)}.
\end{equation}
This should correspond to the diagonal terms.

For the remaining two terms in \eqref{1}, we note that they vanish if $N< 2g$, and if $N\geq 2g$ they contribute
\begin{align}\label{type1}
\frac{q^{-2g\alpha}-q^{-2[N/2]\alpha}}{1-q^{2\alpha}}+\frac{q^{-2g\alpha}}{q-1}.
\end{align}
This should correspond to the Type-I terms. Combining \eqref{type01} and \eqref{type1} we arrive at the following conjecture.

\begin{conjecture}\label{conjecture1level}
We have
\begin{align*}
I_1(N;\alpha)&=\frac{q^{-2[N/2]\alpha}-1}{1-q^{2\alpha}}-B(\alpha)+\mathds{1}_{N\geq 2g}\bigg(\frac{q^{-2g\alpha}-q^{-2[N/2]\alpha}}{1-q^{2\alpha}}+\frac{q^{-2g\alpha}}{q-1}\bigg)\\
&\qquad\qquad+O_\varepsilon(q^{-g+\varepsilon g}) +O_\varepsilon(q^{-N/2+\varepsilon N}).
\end{align*}
\end{conjecture}

\section{The one level density}
\label{1compute}
We assume in this section that $N<4g$.

\subsection{The diagonal}

The diagonal, denoted by $I_1^0(N;\alpha)$, corresponds to the terms $f=P^{2k}$ in \eqref{maineq}, and so in view of Lemma \ref{L4} we have
\[
I_1^0(N;\alpha)=\sum_{1\leq kn\leq [N/2]}\sum_{P\in\mathcal{P}_{n}}\frac{d(P)}{|P|^{k(1+2\alpha)}}-\sum_{1\leq kn\leq [N/2]}\sum_{P\in\mathcal{P}_{n}}\frac{d(P)}{|P|^{k(1+2\alpha)}(|P|+1)}+O_\varepsilon(q^{-2g+\varepsilon g}).
\]
The first term, by the Prime Polynomial Theorem, is equal to
\begin{align*}
\sum_{d(f)\leq [N/2]}\frac{\Lambda(f)}{|f|^{1+2\alpha}}&=\sum_{1\leq n\leq [N/2]}q^{-2n\alpha }=\frac{q^{-2[N/2]\alpha}-1}{1-q^{2\alpha}}.
\end{align*}
For the second term, note that
\begin{align*}
\sum_{1\leq kn\leq [N/2]}\sum_{P\in\mathcal{P}_{n}}\frac{d(P)}{|P|^{k(1+2\alpha)}(|P|+1)}&=B(\alpha)-\sum_{kn> [N/2]}\sum_{P\in\mathcal{P}_{n}}\frac{d(P)}{|P|^{k(1+2\alpha)}(|P|+1)}\\
&=B(\alpha)+O_\varepsilon\big(q^{-N/2+\varepsilon g}\big).
\end{align*}
Hence,
\[
I_1^0(N;\alpha)=\frac{q^{-2[N/2]\alpha}-1}{1-q^{2\alpha}}-B(\alpha)+O_\varepsilon\big(q^{-N/2+\varepsilon g}\big).
\]
Notice that the leading term matches up with \eqref{type01}.

\subsection{Type-I terms}

We now evaluate the off-diagonal terms corresponding to $f=P^{2k+1}$ in \eqref{maineq},
\[
I_1^1(N;\alpha)=\frac{1}{|\mathcal{H}_{2g+1}|}\sum_{d(P^{2k+1})\leq N}\frac{d(P)}{|P|^{(2k+1)(1/2+\alpha)}}\sum_{D\in\mathcal{H}_{2g+1}}\chi_D(P).
\]
Combining the Polya-Vinogradov inequality in Lemma \ref{pv} with the Prime Polynomial Theorem, the contribution of the terms with $k\geq 1$ is
\begin{equation*}
  \ll q^{-2g}\sum_{n\leq N}\sum_{k\geq1}q^{-(k-1)n}\ll Nq^{-2g},
  \end{equation*}
and the contribution of the terms with $d(P)=n$ is
\[
\ll q^{n-2g}.
\] So
\[
I_1^1(N;\alpha)=\frac{1}{|\mathcal{H}_{2g+1}|}\sum_{g+1\leq d(P)\leq N}\frac{d(P)}{|P|^{1/2+\alpha}}\sum_{D\in\mathcal{H}_{2g+1}}\chi_D(P)+O(q^{- g}).
\]

From Lemma \ref{L1} we have
\begin{align*}
&\sum_{D\in\mathcal{H}_{2g+1}}\chi_D(P)=\sum_{C|P^\infty}\sum_{h\in\mathcal{M}_{2g+1-2d(C)}}\chi_P(h)-q\sum_{C|P^\infty}\sum_{h\in\mathcal{M}_{2g-1-2d(C)}}\chi_P(h).
\end{align*}
The sums over $h$ are non-zero only if $0\leq 2g\pm1-2d(C)<d(P) $. Since $C|P^\infty$ and $d(P)\geq g+1$, we must have $C=1$ and, consequently, $d(P)\geq 2g$. Thus,
\[
I_1^1(N;\alpha)=\frac{1}{|\mathcal{H}_{2g+1}|}\sum_{2g\leq d(P)\leq N}\frac{d(P)}{|P|^{1/2+\alpha}}\bigg(\sum_{h\in\mathcal{M}_{2g+1}}\chi_P(h)-q\sum_{h\in\mathcal{M}_{2g-1}}\chi_P(h)\bigg)+O(q^{- g}).
\]

Consider the terms with $d(P)$ odd. Applying Lemma \ref{L2} and Lemma \ref{L3}, the expression inside the bracket is
\[
\frac{q^{2g+3/2}}{|P|^{1/2}}\sum_{d(V)= d(P)-2g-2}\chi_P(V)-\frac{q^{2g+1/2}}{|P|^{1/2}}\sum_{d(V)= d(P)-2g}\chi_P(V).
\]
Notice that $V$ cannot be a square in the sums, and hence by Lemma \ref{sumprimes}, the contribution of these terms to $I_1^1(N;\alpha)$ is $O(Nq^{N/2-2g})$.

If $d(P)$ is even, then from Lemma \ref{L2} and Lemma \ref{L3} we have
\begin{align*}
\sum_{h\in\mathcal{M}_{2g+1}}\chi_P(h)-q\sum_{h\in\mathcal{M}_{2g-1}}\chi_P(h)=&\frac{q^{2g+1}}{|P|^{1/2}}\bigg(q\sum_{d(V)\leq d(P)-2g-3}\chi_P(V)-\sum_{d(V)\leq d(P)-2g-2}\chi_P(V)\bigg)\\
&\ -\frac{q^{2g}}{|P|^{1/2}}\bigg(q\sum_{d(V)\leq d(P)-2g-1}\chi_P(V)-\sum_{d(V)\leq d(P)-2g}\chi_P(V)\bigg).
\end{align*}
As above, the contribution of the terms $V$ non-square is negligible. For $V=\square$, as $d(V)<d(P)$ we have $\chi_P(V)=1$. Thus, the contribution from $V=\square$ is
\begin{align*}
&\frac{q^{2g+1}}{|P|^{1/2}}\bigg(q\sum_{d(V)\leq d(P)/2-g-2}1-\sum_{d(V)\leq d(P)/2-g-1}1\bigg)-\frac{q^{2g}}{|P|^{1/2}}\bigg(q\sum_{d(V)\leq d(P)/2-g-1}1-\sum_{d(V)\leq d(P)/2-g}1\bigg)\\
&\qquad=\begin{cases}
-\frac{q^{2g}(q-1)}{|P|^{1/2}}& \textrm{if }d(P)\geq 2g+2,\\
q^g & \textrm{if }d(P)=2g.
\end{cases}
\end{align*}
We hence obtain that
\[
I_1^1(N;\alpha)=\mathds{1}_{N\geq 2g}\bigg(-\sum_{g+1\leq n\leq[N/2]}\sum_{P\in\mathcal{P}_{2n}}\frac{d(P)}{|P|^{1+\alpha}}+\frac{q^{-2g\alpha}}{q-1}\bigg)+O(Nq^{N/2-2g})+O(q^{- g}).
\]

Now, in view of the Prime Polynomial Theorem,
\begin{align*}
&\sum_{g+1\leq n\leq[N/2]}\sum_{P\in\mathcal{P}_{2n}}\frac{d(P)}{|P|^{1+\alpha}}=\sum_{g+1\leq n\leq[N/2]}q^{-2n(1+\alpha)}\big(q^{2n}+O(q^n)\big)\\
&\qquad\qquad=\sum_{g+1\leq n\leq[N/2]}q^{-2n\alpha}+O(q^{- g})=-\frac{q^{-2g\alpha}-q^{-2[N/2]\alpha}}{1-q^{2\alpha}}+O(q^{- g}).
\end{align*}
So
\[
I_1^1(N;\alpha)=\mathds{1}_{N\geq 2g}\bigg(\frac{q^{-2g\alpha}-q^{-2[N/2]\alpha}}{1-q^{2\alpha}}+\frac{q^{-2g\alpha}}{q-1}\bigg)+O(Nq^{N/2-2g})+O(q^{- g}).
\]
Notice that the leading term matches up with \eqref{type1}.

\section{The two level density - Using the Ratios Conjecture}
\label{2rc}
\subsection{The Ratios Conjecture}

We would like to study
\[
\frac{1}{|\mathcal{H}_{2g+1}|}\sum_{D\in\mathcal{H}_{2g+1}}\frac{L(1/2+\alpha,\chi_D)L(1/2+\beta,\chi_D)}{L(1/2+\gamma,\chi_D)L(1/2+\delta,\chi_D)}
\]
using the recipe in \cite{CS}, where the shifts are assumed to satisfy $|\alpha|, |\beta|, |\gamma|, |\delta| \ll 1/g$. %We will assume for simplicity that $\alpha, \beta>0$ (otherwise we can use the functional equation for the $L$--functions in the numerator). 

We use the approximate functional equation for each of the two $L$--functions in the numerator. The contribution coming from the first parts of the approximate functional equations is equal to
\[
\frac{1}{|\mathcal{H}_{2g+1}|}\sum_{f_1,f_2,h_1,h_2}\frac{\mu(h_1)\mu(h_2)}{|f_1|^{1/2+\alpha}|f_2|^{1/2+\beta}|h_1|^{1/2+\gamma}|h_2|^{1/2+\delta}}\sum_{D\in\mathcal{H}_{2g+1}}\chi_D(f_1f_2h_1h_2).
\]
We only keep the terms with $f_1f_2h_1h_2=\square$. The above expression then becomes
\begin{align*}
&\sum_{f_1f_2h_1h_2=\square}\frac{\mu(h_1)\mu(h_2)a(f_1f_2h_1h_2)}{|f_1|^{1/2+\alpha}|f_2|^{1/2+\beta}|h_1|^{1/2+\gamma}|h_2|^{1/2+\delta}},
\end{align*}
where
\[
a(f)=\prod_{P|f}\bigg(1+\frac{1}{|P|}\bigg)^{-1}.
\]
Using multiplicativity, this is equal to
\begin{align*}
&\prod_{P}\sum_{\substack{f_1,f_2,h_1,h_2\\f_1+f_2+h_1+h_2\ \textrm{even}}}\frac{\mu(P^{h_1})\mu(P^{h_2})a(P^{f_1+f_2+h_1+h_2})}{|P|^{(1/2+\alpha)f_1+(1/2+\beta)f_2+(1/2+\gamma)h_1+(1/2+\delta)h_2}}\\
&\qquad\qquad=A(\alpha,\beta,\gamma,\delta)\frac{\zeta_q(1+2\alpha)\zeta_q(1+2\beta)\zeta_q(1+\alpha+\beta)\zeta_q(1+\gamma+\delta)}{\zeta_q(1+\alpha+\gamma)\zeta_q(1+\alpha+\delta)\zeta_q(1+\beta+\gamma)\zeta_q(1+\beta+\delta)},
\end{align*}
where 
\begin{align*}
&A(\alpha,\beta,\gamma,\delta)=\prod_P\bigg(1+\frac{1}{|P|}\bigg)^{-1}\bigg(1-\frac{1}{|P|^{1+\alpha+\beta}}\bigg)\bigg(1-\frac{1}{|P|^{1+\gamma+\delta}}\bigg)\\
&\qquad \bigg(1-\frac{1}{|P|^{1+\alpha+\gamma}}\bigg)^{-1}\bigg(1-\frac{1}{|P|^{1+\alpha+\delta}}\bigg)^{-1}\bigg(1-\frac{1}{|P|^{1+\beta+\gamma}}\bigg)^{-1}\bigg(1-\frac{1}{|P|^{1+\beta+\delta}}\bigg)^{-1}\\
&\qquad\qquad\bigg(1+\frac{1}{|P|}+\frac{1}{|P|^{1+\alpha+\beta}}+\frac{1}{|P|^{1+\gamma+\delta}}-\frac{1}{|P|^{1+\alpha+\gamma}}-\frac{1}{|P|^{1+\alpha+\delta}}-\frac{1}{|P|^{1+\beta+\gamma}}-\frac{1}{|P|^{1+\beta+\delta}}\\
&\qquad\qquad\qquad-\frac{1}{|P|^{2+2\alpha}}-\frac{1}{|P|^{2+2\beta}}+\frac{1}{|P|^{2+\alpha+\beta+\gamma+\delta}}+\frac{1}{|P|^{3+2\alpha+2\beta}}\bigg).
\end{align*}

The contributions from the other parts of the approximate functional equations can be determined by using the functional equation
\[
L(\tfrac{1}{2}+\alpha,\chi_D)=q^{-2g\alpha }L(\tfrac{1}{2}-\alpha,\chi_D).
\]Hence we have the following.

\begin{conjecture}
We have
\begin{align*}
&\frac{1}{|\mathcal{H}_{2g+1}|}\sum_{D\in\mathcal{H}_{2g+1}}\frac{L(1/2+\alpha,\chi_D)L(1/2+\beta,\chi_D)}{L(1/2+\gamma,\chi_D)L(1/2+\delta,\chi_D)}\\
&\qquad=A(\alpha,\beta,\gamma,\delta)\frac{\zeta_q(1+2\alpha)\zeta_q(1+2\beta)\zeta_q(1+\alpha+\beta)\zeta_q(1+\gamma+\delta)}{\zeta_q(1+\alpha+\gamma)\zeta_q(1+\alpha+\delta)\zeta_q(1+\beta+\gamma)\zeta_q(1+\beta+\delta)}\\
&\qquad\qquad+q^{-2g\alpha }A(-\alpha,\beta,\gamma,\delta)\frac{\zeta_q(1-2\alpha)\zeta_q(1+2\beta)\zeta_q(1-\alpha+\beta)\zeta_q(1+\gamma+\delta)}{\zeta_q(1-\alpha+\gamma)\zeta_q(1-\alpha+\delta)\zeta_q(1+\beta+\gamma)\zeta_q(1+\beta+\delta)}\\
&\qquad\qquad+q^{-2g\beta }A(\alpha,-\beta,\gamma,\delta)\frac{\zeta_q(1+2\alpha)\zeta_q(1-2\beta)\zeta_q(1+\alpha-\beta)\zeta_q(1+\gamma+\delta)}{\zeta_q(1+\alpha+\gamma)\zeta_q(1+\alpha+\delta)\zeta_q(1-\beta+\gamma)\zeta_q(1-\beta+\delta)}\\
&\qquad\qquad+q^{-2g(\alpha+\beta) }A(-\alpha,-\beta,\gamma,\delta)\frac{\zeta_q(1-2\alpha)\zeta_q(1-2\beta)\zeta_q(1-\alpha-\beta)\zeta_q(1+\gamma+\delta)}{\zeta_q(1-\alpha+\gamma)\zeta_q(1-\alpha+\delta)\zeta_q(1-\beta+\gamma)\zeta_q(1-\beta+\delta)}\\
&\qquad\qquad+O_\varepsilon\big(q^{-g+\varepsilon g}\big).
\end{align*}
\end{conjecture}

Notice that for a function $f(u,v)$ analytic at $(u,v)=(r,r)$ and a function $F(s)$ having a simple pole at $s=1$ with residue $r_{F}$, we have
\begin{equation*}\label{trick101}
\frac{\partial}{\partial\alpha}\frac{f(\alpha,\gamma)}{F(1-\alpha+\gamma)}\bigg|_{\alpha=\gamma=r}=-\frac{f(r,r)}{r_{F}}.
\end{equation*}
As $r_{\zeta_q}=1/\log q$, taking derivatives with respect to $\alpha$ and $\beta$, and setting $\gamma=\alpha$, $\delta=\beta$ we obtain

\begin{conjecture}
We have
\begin{align*}
&\frac{1}{|\mathcal{H}_{2g+1}|}\sum_{D\in\mathcal{H}_{2g+1}}\frac{L'}{L}(\tfrac12+\alpha,\chi_D)\frac{L'}{L}(\tfrac12+\beta,\chi_D)\\
&\qquad=\frac{\zeta_q'}{\zeta_q}(1+2\alpha)\frac{\zeta_q'}{\zeta_q}(1+2\beta)+\bigg(\frac{\zeta_q'}{\zeta_q}\bigg)'(1+\alpha+\beta)\nonumber\\
&\qquad\qquad+(\log q)B(\alpha)\frac{\zeta_q'}{\zeta_q}(1+2\beta)+(\log q)B(\beta)\frac{\zeta_q'}{\zeta_q}(1+2\alpha)+(\log q)^2C(\alpha,\beta)\\
&\qquad\qquad+q^{-2g\alpha }(\log q)^2A_2(\alpha)T_2(\alpha,\beta)+q^{-2g\beta }(\log q)^2A_2(\beta)T_2(\beta,\alpha)\\
&\qquad\qquad+q^{-2g(\alpha+\beta) }(\log q)^2A(\alpha,\beta)\frac{\zeta_q(1-2\alpha)\zeta_q(1-2\beta)\zeta_q(1-\alpha-\beta)\zeta_q(1+\alpha+\beta)}{\zeta_q(1-\alpha+\beta)\zeta_q(1+\alpha-\beta)}\\
&\qquad\qquad+O_\varepsilon\big(q^{-g+\varepsilon g}\big),
\end{align*}
where 
\begin{align*}
A_2(\alpha)&:=A(-\alpha,\beta,\alpha,\beta)\zeta_q(1-2\alpha)=\frac{\zeta_q(2)\zeta_q(1-2\alpha)}{\zeta_q(2-2\alpha)}\\
&=\frac{1}{1-q^{2\alpha}}+\frac{1}{q-1}=\frac{q^{2\alpha}}{1-q^{2\alpha}}+\frac{q}{q-1},
\end{align*}
\begin{align*}
A(\alpha,\beta)&:=A(-\alpha,-\beta,\alpha,\beta)\\
&=\prod_{P}\bigg(1+\frac{1}{|P|}\bigg)^{-1}\bigg(1-\frac{1}{|P|}\bigg)^{-2}\bigg(1-\frac{1}{|P|^{1-\alpha-\beta}}\bigg)\bigg(1-\frac{1}{|P|^{1+\alpha+\beta}}\bigg)\\
&\qquad \bigg(1-\frac{1}{|P|^{1-\alpha+\beta}}\bigg)^{-1}\bigg(1-\frac{1}{|P|^{1+\alpha-\beta}}\bigg)^{-1}\bigg(1-\frac{1}{|P|}+\frac{1}{|P|^{1-\alpha-\beta}}+\frac{1}{|P|^{1+\alpha+\beta}}\\
&\qquad\qquad-\frac{1}{|P|^{1-\alpha+\beta}}-\frac{1}{|P|^{1+\alpha-\beta}}-\frac{1}{|P|^{2-2\alpha}}-\frac{1}{|P|^{2-2\beta}}+\frac{1}{|P|^{2}}+\frac{1}{|P|^{3-2\alpha-2\beta}}\bigg),
\end{align*}
$B(\alpha)$ is defined in \eqref{Br}, 
\begin{align*}
C(\alpha,\beta)&=B(\alpha)B(\beta)+\sum_{P}\frac{d(P)^2\big(|P|^{2+\alpha+\beta}(|P|+1)(|P|^\alpha-|P|^\beta)^2-(|P|^{1+\alpha+\beta}-1)^3\big)}{(|P|^{1+2\alpha}-1)(|P|^{1+2\beta}-1)(|P|^{1+\alpha+\beta}-1)^2(|P|+1)}\\
&\qquad -\sum_{P}\frac{d(P)^2}{(|P|^{1+2\alpha}-1)(|P|^{1+2\beta}-1)(|P|+1)^2}\\
&=B(\alpha)B(\beta)+\sum_{P}\frac{d(P)^2|P|^{2+\alpha+\beta}(|P|^\alpha-|P|^\beta)^2}{(|P|^{1+2\alpha}-1)(|P|^{1+2\beta}-1)(|P|^{1+\alpha+\beta}-1)^2}\\
&\qquad-\sum_{P}\frac{d(P)^2|P|^{1+\alpha+\beta}}{(|P|^{1+2\alpha}-1)(|P|^{1+2\beta}-1)(|P|+1)}+\sum_{P}\frac{d(P)^2|P|}{(|P|^{1+2\alpha}-1)(|P|^{1+2\beta}-1)(|P|+1)^2}
\end{align*}
and
\begin{align}\label{factno6}
T_2(\alpha,\beta)&=\frac{1}{\log q}\bigg(\frac{\zeta_q'}{\zeta_q}(1+\alpha+\beta)-\frac{\zeta_q'}{\zeta_q}(1-\alpha+\beta)-\frac{\zeta_q'}{\zeta_q}(1+2\beta)-\frac{\partial A_2(-\alpha,b,\alpha,\beta)/\partial b\big|_{b=\beta}}{A_2(-\alpha,\beta,\alpha,\beta)}\bigg)\nonumber\\
&=\sum_{P}\frac{d(P)(|P|^{2(1-\alpha)}-|P|^{1-2\alpha}-|P|^{2-3\alpha+\beta}+|P|^{2-\alpha+\beta})}{(|P|^{2(1-\alpha)}-1)(|P|^{1+2\beta}-1)}.
\end{align}
%\begin{align*}
%T_2(\alpha,\beta)&=\frac{1}{\log q}\bigg(\frac{\zeta_q'}{\zeta_q}(1+\alpha+\beta)-\frac{\zeta_q'}{\zeta_q}(1-\alpha+\beta)-\frac{\zeta_q'}{\zeta_q}(1+2\beta)-\frac{\partial A_2(-\alpha,b,\alpha,\beta)/\partial b\big|_{b=\beta}}{A_2(-\alpha,\beta,\alpha,\beta)}\bigg)\nonumber\\ 
%&= \sum_P \frac{|P|^{3-3\alpha+\beta}-|P|^{2-2\alpha}+|P|^{1-2\alpha}-|P|^{3-4\alpha+2\beta}+|P|^{3-2\alpha}+|P|^{2-\alpha+\beta}}{(|P|^{2(1-\alpha)}-1)(|P|^{1+2\beta}-1)(|P|^{1-\alpha+\beta}-1)} - \frac{1}{1-q^{\beta-\alpha}}.
%\end{align*}
\end{conjecture}

Equivalently we have

\begin{conjecture}\label{RCF2}
We have
\begin{align*}
&\frac{1}{|\mathcal{H}_{2g+1}|}\sum_{D\in\mathcal{H}_{2g+1}}uv\frac{\mathcal{L}'}{\mathcal{L}}(u,\chi_D)\frac{\mathcal{L}'}{\mathcal{L}}(v,\chi_D)\\
&\qquad=u^2v^2\frac{\mathcal{Z}'}{\mathcal{Z}}(u^2)\frac{\mathcal{Z}'}{\mathcal{Z}}(v^2)+u^2v^2\bigg(\frac{\mathcal{Z}'}{\mathcal{Z}}\bigg)'(uv)+\mathcal{B}(u)v^2\frac{\mathcal{Z}'}{\mathcal{Z}}(v^2)+\mathcal{B}(v)u^2\frac{\mathcal{Z}'}{\mathcal{Z}}(u^2)+\mathcal{C}(u,v)\\
&\qquad\qquad+(qu^2)^g\mathcal{A}_2(u)\mathcal{T}_2(u,v)+(qv^2)^g\mathcal{A}_2(v)\mathcal{T}_2(v,u)\\
&\qquad\qquad+(quv)^{2g}\mathcal{A}(u,v)\frac{\mathcal{Z}\big(\frac{1}{q^2u^2}\big)\mathcal{Z}\big(\frac{1}{q^2v^2}\big)\mathcal{Z}\big(\frac{1}{q^2uv}\big)\mathcal{Z}(uv)}{\mathcal{Z}\big(\frac{v}{qu}\big)\mathcal{Z}\big(\frac{u}{qv}\big)}+O_\varepsilon\big(q^{-g+\varepsilon g}\big),
\end{align*}
where 
\begin{align*}
\mathcal{A}_2(u)&=\frac{qu^2}{qu^2-1}+\frac{1}{q-1}=\frac{1}{qu^2-1}+\frac{q}{q-1},
\end{align*}
\begin{align*}
\mathcal{A}(u,v)&=\prod_{P}\bigg(1+\frac{1}{|P|}\bigg)^{-1}\bigg(1-\frac{1}{|P|}\bigg)^{-2}\bigg(1-\frac{1}{|P|^{2}(uv)^{d(P)}}\bigg)\Big(1-(uv)^{d(P)}\Big)\\
&\qquad \bigg(1-\frac{v^{d(P)}}{|P|u^{d(P)}}\bigg)^{-1}\bigg(1-\frac{u^{d(P)}}{|P|v^{d(P)}}\bigg)^{-1}\bigg(1-\frac{1}{|P|}+\frac{1}{|P|^{2}(uv)^{d(P)}}+(uv)^{d(P)}\\
&\qquad\qquad-\frac{v^{d(P)}}{|P|u^{d(P)}}-\frac{u^{d(P)}}{|P|v^{d(P)}}-\frac{1}{|P|^{3}u^{2d(P)}}-\frac{1}{|P|^{3}v^{2d(P)}}+\frac{1}{|P|^{2}}+\frac{1}{|P|^{5}(uv)^{2d(P)}}\bigg),
\end{align*}
$\mathcal{B}(u)$ is defined in \eqref{mathcalBr}, 
\begin{align*}
\mathcal{C}(u,v)&=\mathcal{B}(u)\mathcal{B}(v)+\sum_{P}\frac{d(P)^2(uv)^{d(P)}(u^{d(P)}-v^{d(P)})^2}{(1-u^{2d(P)})(1-v^{2d(P)})(1-(uv)^{d(P)})^2}\\
&\qquad-\sum_{P}\frac{d(P)^2(uv)^{d(P)}}{(1-u^{2d(P)})(1-v^{2d(P)})(|P|+1)}+\sum_{P}\frac{d(P)^2|P|(uv)^{2d(P)}}{(1-u^{2d(P)})(1-v^{2d(P)})(|P|+1)^2}
\end{align*}
and
\begin{align*}
\mathcal{T}_2(u,v)&=\sum_{P}\frac{d(P)(|P|^{3}(uv)^{2d(P)}-|P|^{2}(uv)^{2d(P)}-|P|^{3}u^{3d(P)}v^{d(P)}+|P|^{2}(uv)^{d(P)})}{(|P|^{3}u^{2d(P)}-1)(1-v^{2d(P)})}.
\end{align*}
%\acom{
%\begin{align*}
%\mathcal{T}_2(u,v) &= \sum_P\Big( |P|^4 u^{3d(P)}v^{2d(P)}-|P|^3u^{2d(P)}v^{3d(P)}+|P|^2u^{2d(P)}v^{3d(P)}-|P|^4 u^{4d(P)}v^{d(P)} \\
%&+|P|^4 u^{2d(P)}v^{3d(P)}+|P|^2u^{d(P)} v^{2d(P)}\Big) \Big/ \Big((|P|^3u^{2d(P)}-1)(1-v^{2d(P)})(|P|u^{d(P)}-v^{d(P)}) \Big) \\
%&-\frac{1}{1-q^{\beta-\alpha}}.
%\end{align*}
%}
\end{conjecture}

\subsection{The two level density}

Consider
\begin{equation}\label{2level}
I_2(N;\alpha,\beta)=\frac{1}{|\mathcal{H}_{2g+1}|}\sum_{D\in\mathcal{H}_{2g+1}}\sum_{\substack{f_1,f_2\in\mathcal{M}\\d(f_1f_2)\leq N}}\frac{\Lambda(f_1)\Lambda(f_2)\chi_D(f_1f_2)}{|f|_1^{1/2+\alpha}|f|_2^{1/2+\beta}}.
\end{equation}

Using the Perron formula \eqref{Perron} this is equal to
\begin{align*}
&\frac{1}{|\mathcal{H}_{2g+1}|}\sum_{D\in\mathcal{H}_{2g+1}}\frac{1}{2\pi i}\oint_{|u|=r}\sum_{f_1,f_2\in\mathcal{M}}\frac{\Lambda(f_1)\Lambda(f_2)\chi_D(f_1f_2)u^{d(f_1)+d(f_2)}}{|f|_1^{1/2+\alpha}|f|_2^{1/2+\beta}}\frac{du}{u^{N+1}(1-u)}\\
&\qquad=\frac{1}{|\mathcal{H}_{2g+1}|}\sum_{D\in\mathcal{H}_{2g+1}}\frac{1}{2\pi i}\oint_{|u|=r} \frac{u^2}{q^{1+\alpha+\beta}} \frac{ \mathcal{L}'}{\mathcal{L}} \Big(\frac{u}{q^{1/2+\alpha}},\chi_D \Big)  \frac{ \mathcal{L}'}{\mathcal{L}} \Big(\frac{u}{q^{1/2+\beta}},\chi_D \Big)\frac{du}{u^{N+1}(1-u)}
\end{align*}
for any $r<q^{-1/2-\varepsilon}$.  We enlarge to contour to $|u|=r=q^{-\varepsilon}$.
In view of Conjecture \ref{RCF2} we write
\begin{align}
I_2(N;\alpha,\beta)=\frac{1}{2\pi i}\oint_{|u|=r}\sum_{j=1}^{4}R_j(u,\alpha,\beta)\,\frac{du}{u^{N+1}(1-u)}+O_\varepsilon\big(q^{-g+\varepsilon g}\big).
\label{int_density}
\end{align}
The terms coming from the first parts of the approximate functional equations, $R_1(u,\alpha,\beta)$, correspond to the diagonal terms, while the terms coming from only $1$ swap in the approximate functional equations, $R_2(u,\alpha,\beta)$ and $R_3(u,\alpha,\beta)$, correspond to the Type-I terms. Type-II terms are the terms with $2$ swaps, $R_4(u,\alpha,\beta)$.

For the $0$ swap terms we have
\begin{align}\label{type02level}
R_1(u,\alpha,\beta)&=\frac{u^4}{q^{2(1+\alpha+\beta)}}\frac{\mathcal{Z}'}{\mathcal{Z}} \Big(\frac{u^2}{q^{1+2\alpha}}\Big)\frac{\mathcal{Z}'}{\mathcal{Z}}\Big(\frac{u^2}{q^{1+2\beta}}\Big)+\frac{u^4}{q^{2(1+\alpha+\beta)}}\bigg(\frac{\mathcal{Z}'}{\mathcal{Z}}\bigg)'\Big(\frac{u^2}{q^{1+\alpha+\beta}}\Big)\\
&\qquad\qquad+\mathcal{B}(u,\alpha) \frac{u^2}{q^{1+2\beta}}\frac{\mathcal{Z}'}{\mathcal{Z}}\Big(\frac{u^2}{q^{1+2\beta}}\Big)+\mathcal{B}(u,\beta)\frac{u^2}{q^{1+2\alpha}}\frac{\mathcal{Z}'}{\mathcal{Z}}\Big(\frac{u^2}{q^{1+2\alpha}}\Big)+\mathcal{C}(u,\alpha,\beta),\nonumber
\end{align}
where
\begin{align}\label{type02level1}
\mathcal{C}(u,\alpha,\beta)&:=\mathcal{C}\Big(\frac{u}{q^{1/2+\alpha}},\frac{u}{q^{1/2+\beta}}\Big)\nonumber\\
&=\mathcal{B}(u,\alpha)\mathcal{B}(u,\beta)+\sum_{P}\frac{d(P)^2|P|^{2+\alpha+\beta}u^{4d(P)}(|P|^\alpha-|P|^\beta)^2}{(|P|^{1+2\alpha}-u^{2d(P)})(|P|^{1+2\beta}-u^{2d(P)})(|P|^{1+\alpha+\beta}-u^{2d(P)})^2}\nonumber\\
&\qquad\qquad-\sum_{P}\frac{d(P)^2|P|^{1+\alpha+\beta}u^{2d(P)}}{(|P|^{1+2\alpha}-u^{2d(P)})(|P|^{1+2\beta}-u^{2d(P)})(|P|+1)}\\
&\qquad\qquad+\sum_{P}\frac{d(P)^2|P|u^{4d(P)}}{(|P|^{1+2\alpha}-u^{2d(P)})(|P|^{1+2\beta}-u^{2d(P)})(|P|+1)^2}.\nonumber
\end{align}

Concerning the $1$ swap terms we have
\begin{align} \label{r2}
R_2(u,\alpha,\beta)+R_3(u,\alpha,\beta)=q^{-2g\alpha}u^{2g}\mathcal{A}_2(u,\alpha)\mathcal{T}_2(u,\alpha,\beta)+q^{-2g\beta}u^{2g}\mathcal{A}_2(u,\beta)\mathcal{T}_2(u,\beta,\alpha),
\end{align}
where
\begin{align}\label{A22}
\mathcal{A}_2(u,\alpha)&=\mathcal{A}_2\Big(\frac{u}{q^{1/2+\alpha}}\Big)\nonumber\\
&=\frac{u^2}{u^2-q^{2\alpha}}+\frac{1}{q-1}=\frac{q^{2\alpha}}{u^2-q^{2\alpha}}+\frac{q}{q-1}
\end{align}
and
\begin{align*}
\mathcal{T}_2(u,\alpha,\beta)&=\mathcal{T}_2\Big(\frac{u}{q^{1/2+\alpha}},\frac{u}{q^{1/2+\beta}}\Big)\nonumber\\
&=\sum_{P} \frac{u^{2d(P)}(|P|^{2(1-\alpha)}u^{2d(P)}-|P|^{1-2\alpha}u^{2d(P)}-|P|^{2-3\alpha+\beta}u^{2d(P)}+|P|^{2-\alpha+\beta})}{(|P|^{2(1-\alpha)}u^{2d(P)}-1)(|P|^{1+2\beta}-u^{2d(P)})}.
\end{align*}

Note that $1$ swap terms kick in once $N \geq 2g$. In the computation of Type-I terms in section \ref{type1} we also assume that $N  <4g$. We write $\mathcal{T}_2(u,\alpha,\beta)$ as a sum of four terms. For the first three terms, we claim that we can truncate the sum over $P$ to those primes $P$ with $d(P)<g$; otherwise the corresponding integrals in equation \eqref{int_density} will be equal to zero. Indeed, in order for the integrals to be non-vanishing, we need $2g+2d(P)=N$. Since $N<4g$ it follows that $d(P)<g$. We write the fourth term in the expression of $\mathcal{T}_2(u,\alpha,\beta)$ as
\begin{align}
 \sum_P  & \frac{u^{2d(P)} |P|^{2-\alpha+\beta}}{(|P|^{2(1-\alpha)}u^{2d(P)}-1)(|P|^{1+2\beta}-u^{2d(P)})} = \sum_{d(P) <g} \frac{u^{2d(P)} |P|^{2-\alpha+\beta}}{(|P|^{2(1-\alpha)}u^{2d(P)}-1)(|P|^{1+2\beta}-u^{2d(P)})} \label{nonconv} \\
 &+ \sum_{d(P) \geq g} \frac{u^{2d(P)} |P|^{2-\alpha+\beta}}{(|P|^{2(1-\alpha)}u^{2d(P)}-1)(|P|^{1+2\beta}-u^{2d(P)})} \nonumber \\
 &= \sum_{d(P) <g} \frac{u^{2d(P)} |P|^{2-\alpha+\beta}}{(|P|^{2(1-\alpha)}u^{2d(P)}-1)(|P|^{1+2\beta}-u^{2d(P)})} +  \sum_{d(P) \geq g} \frac{1}{|P|^{1+\beta-\alpha}} \nonumber \\
 & + \sum_{d(P) \geq g} \frac{|P|^{2(1-\alpha)} u^{4d(P)}+|P|^{1+2\beta}+u^{2d(P)}}{|P|^{1+\beta-\alpha} (|P|^{2(1-\alpha)}u^{2d(P)}-1)(|P|^{1+2\beta}-u^{2d(P)})} \nonumber \\
 &= \sum_{d(P) <g} \frac{u^{2d(P)} |P|^{2-\alpha+\beta}}{(|P|^{2(1-\alpha)}u^{2d(P)}-1)(|P|^{1+2\beta}-u^{2d(P)})} +  \sum_{d(P) \geq  g} \frac{1}{|P|^{1+\beta-\alpha}}+O(q^{-g}). \nonumber 
 \end{align}
We use the Prime Polynomial Theorem for the sum over $d(P)\geq g$ above and without worrying abut convergence issues since the recipe is a heuristic argument, we replace it by what we get by summing the geometric series. 
Then when $N<4g$ we rewrite
\begin{align}
 \mathcal{T}_2(u,\alpha,\beta) &= \sum_{d(P)<g} \frac{u^{2d(P)}(|P|^{2(1-\alpha)}u^{2d(P)}-|P|^{1-2\alpha}u^{2d(P)}-|P|^{2-3\alpha+\beta}u^{2d(P)}+|P|^{2-\alpha+\beta})}{(|P|^{2(1-\alpha)}u^{2d(P)}-1)(|P|^{1+2\beta}-u^{2d(P)})} \nonumber  \\
 &+ q^{g(-\beta+\alpha)} \frac{1}{q^{\alpha-\beta}-1}.
 \label{t2}
 \end{align}
We remark that although the term in the second line above gives a term involving $q^{-g(\alpha+\beta)}$ in the expression of $R_2(u,\alpha,\beta)$, when we put all the terms together, the contributions of this type will cancel out.
%The integral of $R_2(u,\alpha,\beta)+R_3(u,\alpha,\beta)$ as in \eqref{int_density} gives Type-I terms.

For the $2$ swaps terms we have
\begin{align*}
R_4(u,\alpha,\beta) = q^{-2g(\alpha+\beta)} u^{4g} \mathcal{A} \Big( \frac{u}{q^{1/2+\alpha}}, \frac{u}{q^{1/2+\beta}}\Big) \frac{ \mathcal{Z}( \frac{1}{q^{1-2 \alpha}u^2}) \mathcal{Z}( \frac{1}{q^{1-2 \beta}u^2} ) \mathcal{Z}( \frac{1}{q^{1- \alpha-\beta}u^2} ) \mathcal{Z}( \frac{u^2}{q^{1+\alpha+\beta}} )}{\mathcal{Z} (\frac{1}{q^{1-\alpha+\beta}}) \mathcal{Z}(\frac{1}{q^{1+\alpha-\beta}})}.
\end{align*}

\section{The two level density - The diagonal}
\label{2diag}
In this and the following section, we assume that $N<4g$.

The diagonal, denoted by $I_2^0(N;\alpha,\beta)$, comes of the terms with $f_1f_2=\square$ in \eqref{2level}. From Lemma \ref{L4} and the Perron formula \eqref{Perron} we have
\begin{align}\label{600}
I_2^0(N;\alpha,\beta)&=\sum_{\substack{f_1,f_2\in\mathcal{M}\\d(f_1f_2)\leq N\\ f_1f_2=\square}}\frac{\Lambda(f_1)\Lambda(f_2)}{|f|_1^{1/2+\alpha}|f|_2^{1/2+\beta}}\prod_{P|f_1f_2}\bigg(1-\frac{1}{|P|+1}\bigg)+O_\varepsilon(q^{-2g+\varepsilon g})\nonumber\\
&=\frac{1}{2\pi i}\oint_{|u|=r}J_2^0(u,\alpha,\beta)\frac{du}{u^{N+1}(1-u)}+O_\varepsilon(q^{-2g+\varepsilon g})
\end{align}
for any $r<q^{-1/2-\varepsilon}$, where
\[
J_2^0(u,\alpha,\beta)=\sum_{\substack{f_1,f_2\in\mathcal{M}\\f_1f_2=\square}}\frac{\Lambda(f_1)\Lambda(f_2)u^{d(f_1f_2)}}{|f|_1^{1/2+\alpha}|f|_2^{1/2+\beta}}\prod_{P|f_1f_2}\bigg(1-\frac{1}{|P|+1}\bigg).
\]
We write
\[
J_2^0(u,\alpha,\beta)=J_{2}^{0,\textrm{ee}}(u,\alpha,\beta)+J_{2}^{0,\textrm{oo}}(u,\alpha,\beta),
\]
where $J_{2}^{0,\textrm{ee}}(u,\alpha,\beta)$ consists of the terms $f_1=P^{2k}$, $f_2=Q^{2l}$ with $k,l\geq1$, and $J_{2}^{0,\textrm{oo}}(u,\alpha,\beta)$ consists of the terms $f_1=P^{2k+1}$, $f_2=P^{2l+1}$ with $k,l\geq0$.

We have
\begin{align}\label{Jee}
J_{2}^{0,\textrm{ee}}(u,\alpha,\beta)&=\sum_{k,l\geq1}\sum_{P, Q}\frac{d(P)d(Q)u^{2kd(P)+2ld(Q)}}{|P|^{k(1+2\alpha)}|Q|^{l(1+2\beta)}}\bigg(1-\frac{1}{|P|+1}\bigg)\bigg(1-\frac{1}{|Q|+1}\bigg)\nonumber\\
&\qquad\qquad+\sum_{k,l\geq1}\sum_{P}\frac{d(P)^2|P|u^{2(k+l)d(P)}}{|P|^{k(1+2\alpha)+l(1+2\beta)}(|P|+1)^2}\nonumber\\
&=\bigg(\frac{u^2}{q^{1+2\alpha}}\frac{\mathcal{Z}'}{\mathcal{Z}} \Big(\frac{u^2}{q^{1+2\alpha}}\Big)+\mathcal{B}(u,\alpha)\bigg)\bigg(\frac{u^2}{q^{1+2\beta}}\frac{\mathcal{Z}'}{\mathcal{Z}} \Big(\frac{u^2}{q^{1+2\beta}}\Big)+\mathcal{B}(u,\beta)\bigg)\\
&\qquad\qquad+\sum_{P}\frac{d(P)^2|P|u^{4d(P)}}{(|P|^{1+2\alpha}-u^{2d(P)})(|P|^{1+2\beta}-u^{2d(P)})(|P|+1)^2}.\nonumber
\end{align}
On the other hand,
\begin{align*}
J_{2}^{0,\textrm{oo}}(u,\alpha,\beta)&=\sum_{k,l\geq0}\sum_{P}\frac{d(P)^2u^{2(k+l+1)d(P)}}{|P|^{1+\alpha+\beta+k(1+2\alpha)+l(1+2\beta)}}\bigg(1-\frac{1}{|P|+1}\bigg)\\
&=\sum_{P}\frac{d(P)^2|P|^{1+\alpha+\beta}u^{2d(P)}}{(|P|^{1+2\alpha}-u^{2d(P)})(|P|^{1+2\beta}-u^{2d(P)})}\\
&\qquad\qquad-\sum_{P}\frac{d(P)^2|P|^{1+\alpha+\beta}u^{2d(P)}}{(|P|^{1+2\alpha}-u^{2d(P)})(|P|^{1+2\beta}-u^{2d(P)})(|P|+1)}.
\end{align*}
Note that
\[
u^2\bigg(\frac{\mathcal{Z}'}{\mathcal{Z}}\bigg)'(u)=\sum_{P}\frac{d(P)^2u^{d(P)}}{(1-u^{d(P)})^2}.
\]
So
\begin{align}\label{Joo}
&J_{2}^{0,\textrm{oo}}(u,\alpha,\beta)-\frac{u^4}{q^{2(1+\alpha+\beta)}}\bigg(\frac{\mathcal{Z}'}{\mathcal{Z}}\bigg)'\Big(\frac{u^2}{q^{1+\alpha+\beta}}\Big)\nonumber\\
&\qquad\qquad=\sum_{P}\frac{d(P)^2|P|^{2+\alpha+\beta}u^{4d(P)}(|P|^\alpha-|P|^\beta)^2}{(|P|^{1+2\alpha}-u^{2d(P)})(|P|^{1+2\beta}-u^{2d(P)})(|P|^{1+\alpha+\beta}-u^{2d(P)})^2}\\
&\qquad\qquad\qquad\qquad-\sum_{P}\frac{d(P)^2|P|^{1+\alpha+\beta}u^{2d(P)}}{(|P|^{1+2\alpha}-u^{2d(P)})(|P|^{1+2\beta}-u^{2d(P)})(|P|+1)}.\nonumber
\end{align}
We enlarge the contour in \eqref{600} to $|u|=r=q^{-\varepsilon}$. Combining \eqref{Jee} and \eqref{Joo}, and comparing with \eqref{type02level} and \eqref{type02level1} we see that
\[
J_2^0(u,\alpha,\beta)=R_1(u,\alpha,\beta).
\]

\section{The two level density - Type-I terms}
\label{type1}
\subsection{The terms $f_1=P^{2k+1}$, $f_2=Q^{2l}$ with $k\geq0$, $l\geq1$} \label{type11} We denote this contribution by $I_{2}^{\textrm{oe}}(N;\alpha,\beta)$. In this section, we assume $N \geq 2g$. We have
\begin{align*}
I_{2}^{\textrm{oe}}(N;\alpha,\beta)=&\frac{1}{|\mathcal{H}_{2g+1}|}\sum_{\substack{P\ne Q\\d(P^{2k+1}Q^{2l})\leq N}}\frac{d(P)d(Q)}{|P|^{(2k+1)(1/2+\alpha)}|Q|^{l(1+2\beta)}}\sum_{D\in\mathcal{H}_{2g+1}}\chi_D(PQ^2)\\
&\qquad\qquad+\frac{1}{|\mathcal{H}_{2g+1}|}\sum_{d(P^{2k+2l+1})\leq N}\frac{d(P)^2}{|P|^{(2k+1)(1/2+\alpha)+l(1+2\beta)}}\sum_{D\in\mathcal{H}_{2g+1}}\chi_D(P).
\end{align*}
By the Polya-Vinogradov inequality in Lemma \ref{pv}, the second term is $O(N^2q^{-2g})$. We now consider the first term with $P\ne Q$. The same argument also shows the terms with $k\geq1$ are bounded by the same error term, and the contribution of the terms with $d(P)=n$ is $O(Nq^{n-2g})$. %\acom{I don't think we can use the exact same argument (i.e: Polya-Vinogradov) for $P \neq Q$, because we have coprimality conditions. So I think we need something like this. 
%\begin{lemma}
%For $Q$ a prime, we have
%$$ \sum_{\substack{D \in \mathcal{H}_{2g+1} \\ (D,Q)=1}} \chi_D(P) \ll |P|^{1/2} \frac{g}{\deg(Q)}.$$
%\end{lemma}
%\begin{proof}
%We have
%$$\sum_{\substack{D \in \mathcal{H}_{2g+1} \\ (D,Q)=1}} \chi_D(P) = \sum_{D \in \mathcal{H}_{2g+1}} \chi_D(P) - \sum_{\substack{D \in \mathcal{H}_{2g+1-d(Q)} \\ (D,Q)=1}} \chi_D(P) \chi_Q(P).$$ We repeat this argument and get that
%$$ \sum_{\substack{D \in \mathcal{H}_{2g+1} \\ (D,Q)=1}} \chi_D(P) = \sum_{j=0}^k (-1)^j \chi_Q(P)^j \sum_{D \in \mathcal{H}_{2g+1-jd(Q)}} \chi_D(P),$$ where $k = [(2g+1)/d(Q)].$
%Applying the Polya-Vinogradov bound in Lemma \ref{pv} for the sum over $D$, the conclusion follows.
%\end{proof}
%}
So 
\begin{align*}
I_{2}^{\textrm{oe}}(N;\alpha,\beta)=&\frac{1}{|\mathcal{H}_{2g+1}|}\sum_{\substack{d(PQ^{2l})\leq N\\d(P)\geq g+1}}\frac{d(P)d(Q)}{|P|^{1/2+\alpha}|Q|^{l(1+2\beta)}}\sum_{D\in\mathcal{H}_{2g+1}}\chi_D(PQ^2)+O(Nq^{-g}).
\end{align*}

Applying Lemma \ref{L1} and since $d(P) \geq g+1$, we have
\begin{align}\label{P&Q}
 \sum_{D \in \mathcal{H}_{2g+1}} \chi_D(PQ^2) &= \sum_{j\geq 0} \bigg(\sum_{h \in \mathcal{M}_{2g+1-2jd(Q)}} \chi_{PQ^2}(h)  - q \sum_{h \in \mathcal{M}_{2g-1-2jd(Q)}} \chi_{PQ^2}(h)\bigg).
 \end{align}
If $d(P)$ is odd, then using Lemmas \ref{L2} and \ref{L3} it follows that the term in parenthesis is equal to
\begin{align*}
&\frac{q^{2g+3/2}}{|P|^{1/2}|Q|^{2j+2}}\sum_{d(V)= d(P)+(2j+2)d(Q)-2g-2}\chi_P(V)G(V,Q^2)\\
&\qquad\qquad-\frac{q^{2g+1/2}}{|P|^{1/2}|Q|^{2j+2}}\sum_{d(V)= d(P)+(2j+2)d(Q)-2g}\chi_P(V)G(V,Q^2).
\end{align*}
As $V$ cannot be a square in the sums, by Lemma \ref{sumprimes}, the contribution of these terms to $I_{2}^{\textrm{oe}}(N;\alpha,\beta)$ is $O(N^2q^{N/2-2g})$.

Now consider the case $d(P)$ is even. Applying Lemmas \ref{L2} and \ref{L3}, the first sum over $h$ in \eqref{P&Q} is
\begin{align*}
&\frac{q^{2g+1}}{|P||Q|^{2j+2}}\bigg(q\sum_{V\in\mathcal{M}_{\leq d(P)+(2j+2)d(Q)-2g-3}} G(V,P)G(V,Q^2)\\
&\qquad\qquad\qquad\qquad\qquad\qquad\qquad\qquad- \sum_{V\in\mathcal{M}_{\leq d(P)+(2j+2)d(Q)-2g-2}}G(V,P)G(V,Q^2) \bigg)\\
 &\qquad =-\frac{q^{2g+1}\chi_P(Q)}{|P|^{1/2}|Q|^{2j+1}}\bigg(q\sum_{\substack{V\in\mathcal{M}_{\leq d(P)+(2j+1)d(Q)-2g-3}\\(V,Q)=1}} \chi_P(V)- \sum_{\substack{V\in\mathcal{M}_{\leq d(P)+(2j+1)d(Q)-2g-2}\\(V,Q)=1}}\chi_P(V)\bigg)\\
&\qquad\qquad\qquad+\frac{q^{2g+1}\varphi(Q^2)}{|P|^{1/2}|Q|^{2j+2}}\bigg(q\sum_{V\in\mathcal{M}_{\leq d(P)+2jd(Q)-2g-3}} \chi_P(V)- \sum_{V\in\mathcal{M}_{\leq d(P)+2jd(Q)-2g-2}}\chi_P(V)\bigg).
\end{align*} 
As above, the contribution of the first term and that of $V\ne\square$ in the second term to $I_{2}^{\textrm{oe}}(N;\alpha,\beta)$ is bounded by $O(N^2q^{N/2-2g})$.  We are thus left with $V=\square$ in the second term above, which is equal to
\begin{align*}
&\frac{q^{2g+1}\varphi(Q^2)}{|P|^{1/2}|Q|^{2j+2}}\bigg(q\sum_{\substack{V\in\mathcal{M}_{\leq d(P)/2+jd(Q)-g-2}}} 1- \sum_{\substack{V\in\mathcal{M}_{\leq d(P)/2+jd(Q)-g-1}}}1\bigg)\\
&\qquad=\begin{cases}
-\frac{q^{2g+1}\varphi(Q^2)}{|P|^{1/2}|Q|^{2j+2}} & \quad\textrm{if } d(P)+2jd(Q)>2g,\\
0 & \quad\textrm{otherwise}.
\end{cases}
\end{align*}

The same argument applies to the second sum over $h$ in \eqref{P&Q}, and hence we obtain
\begin{align*}
&I_{2}^{\textrm{oe}}(N;\alpha,\beta)=-\sum_{\substack{d(PQ^{2l})\leq N\\d(P)\, \textrm{even}\,\geq g+1}}\,\sum_{\substack{d(P)+2jd(Q)> 2g}}\frac{d(P)d(Q)\varphi(Q^2)}{|P|^{1+\alpha}|Q|^{l(1+2\beta)+2j+2}}\\
&\qquad+\frac{1}{q-1}\sum_{\substack{d(PQ^{2l})\leq N\\d(P)\geq g+1}}\,\,\sum_{d(P)+2jd(Q)=2g}\frac{d(P)d(Q)\varphi(Q^2)}{|P|^{1+\alpha}|Q|^{l(1+2\beta)+2j+2}}+O(N^2q^{N/2-2g})+O(Nq^{-g}).
\end{align*}
By the Prime Polynomial Theorem, the condition $d(P)\geq g+1$ can be removed at the cost of an error of size $O(gq^{-g/2})$. The same argument also implies that we can restrict the sum to $jd(Q)<g$. So
\begin{align}
I_{2}^{\textrm{oe}}(N;\alpha,\beta) =\frac{1}{2\pi i}\oint_{|u|=r}J_{2}^{\textrm{oe}}(u,\alpha,\beta)\,\frac{du}{u^{N+1}(1-u)}+O(N^2q^{N/2-2g})+O(gq^{-g/2})
\label{i2}
\end{align}
for any $r<q^{-\varepsilon}$, where
\begin{align*}
J_{2}^{\textrm{oe}}(u,\alpha,\beta)&=-\sum_{l\geq1}\sum_{d(P)\, \textrm{even}}\,\sum_{\substack{d(P)+2jd(Q)>2g\\jd(Q)<g}}\frac{d(P)d(Q)\varphi(Q^2)u^{d(P)+2ld(Q)}}{|P|^{1+\alpha}|Q|^{l(1+2\beta)+2j+2}}\\
&\qquad\qquad+\frac{1}{q-1}\sum_{l\geq1}\,\sum_{\substack{d(P)+2jd(Q)=2g}}\frac{d(P)d(Q)\varphi(Q^2)u^{d(P)+2ld(Q)}}{|P|^{1+\alpha}|Q|^{l(1+2\beta)+2j+2}}\\
&=-\sum_{d(P)\, \textrm{even}}\,\sum_{\substack{d(P)+2jd(Q)>2g\\jd(Q)<g}}\frac{d(P)d(Q)\varphi(Q^2)u^{d(P)+2d(Q)}}{|P|^{1+\alpha}|Q|^{2j+2}(|Q|^{1+2\beta}-u^{2d(Q)})}\\
&\qquad\qquad+\frac{1}{q-1}\,\sum_{\substack{d(P)+2jd(Q)=2g}}\frac{d(P)d(Q)\varphi(Q^2)u^{d(P)+2d(Q)}}{|P|^{1+\alpha}|Q|^{2j+2}(|Q|^{1+2\beta}-u^{2d(Q)})}.
\end{align*}

From the Prime Polynomial Theorem we have
\begin{align*}
\sum_{d(P)\, \textrm{even}\,>2g-2jd(Q)}\frac{d(P)u^{d(P)}}{|P|^{1+\alpha}}&=\sum_{n>g-jd(Q)}\frac{u^{2n}}{q^{2n\alpha}}\big(1+O(q^{-n})\big)\\
&=-q^{-2g\alpha}u^{2g-2jd(Q)}|Q|^{2j\alpha}\frac{u^{2}}{u^2-q^{2\alpha}}+O(q^{-g}|Q|^j).
\end{align*}
\kommentar{\acom{For the error term to have that size I think we need to assume $\alpha>0$.}\Hung{We shall assume $\alpha\ll 1/g$, and I think the above error term is fine in that case.}}
Hence, using \eqref{A22}, we get
\begin{align}
J_{2}^{\textrm{oe}}(u,\alpha,\beta)&=q^{-2g\alpha}u^{2g}\mathcal{A}_2(u,\alpha)\sum_{j\geq0}\sum_{Q\in\mathcal{P}}\frac{d(Q)\varphi(Q^2)u^{-2jd(Q)+2d(Q)}}{|Q|^{2j(1-\alpha)+2}(|Q|^{1+2\beta}-u^{2d(Q)})}+O(gq^{-g})\nonumber\\
&=q^{-2g\alpha}u^{2g}\mathcal{A}_2(u,\alpha)\sum_{Q\in\mathcal{P}}\frac{d(Q)(|Q|^{2(1-\alpha)}-|Q|^{1-2\alpha})u^{4d(Q)}}{(|Q|^{2(1-\alpha)}u^{2d(Q)}-1)(|Q|^{1+2\beta}-u^{2d(Q)})}+O(gq^{-g}),
\end{align}
where in the first line we have removed the condition $jd(Q)<g$ with an admissible error.
Note that we can truncate the sum over $Q$ above to $d(Q)<g$ using a similar argument as in section \ref{2rc}. Indeed, when $d(Q) \geq g$ the corresponding term in integral \eqref{i2} will be equal to zero since there will be no poles inside the contour of integration. Then we rewrite
\begin{equation} \label{factno1}
J_{2}^{\textrm{oe}}(u,\alpha,\beta)= q^{-2g\alpha}u^{2g}\mathcal{A}_2(u,\alpha)\sum_{d(Q)<g}\frac{d(Q)(|Q|^{2(1-\alpha)}-|Q|^{1-2\alpha})u^{4d(Q)}}{(|Q|^{2(1-\alpha)}u^{2d(Q)}-1)(|Q|^{1+2\beta}-u^{2d(Q)})}+O(gq^{-g}).
\end{equation}

\subsection{The terms $f_1=P^{2k+1}$, $f_2=Q^{2l+1}$ with $P\ne Q$ and $k,l\geq0$} \label{type12}
We denote
\begin{align*}
&\frac{1}{|\mathcal{H}_{2g+1}|}\sum_{\substack{P\ne Q\\d(P^{2k+1}Q^{2l+1})\leq N}}\frac{d(P)d(Q)}{|P|^{(2k+1)(1/2+\alpha)}|Q|^{(2l+1)(1/2+\beta)}}\sum_{D\in\mathcal{H}_{2g+1}}\chi_D(PQ)\\
&\qquad\qquad=I_{2,>}^{\textrm{oo}}(N;\alpha,\beta)+I_{2,<}^{\textrm{oo}}(N;\alpha,\beta)+I_{2,=}^{\textrm{oo}}(N;\alpha,\beta),
\end{align*}
 corresponding to the terms with $d(P)> d(Q)$, $d(P)<d(Q)$ and  $d(P)=d(Q)$, respectively. 

Applying Lemma \ref{L1} we have
\begin{align}\label{condition}
 \sum_{D \in \mathcal{H}_{2g+1}} \chi_D(PQ) &= \sum_{i,j\geq0} \bigg(\sum_{h \in \mathcal{M}_{2g+1-2id(P)-2jd(Q)}} \chi_{PQ}(h)  - q \sum_{h \in \mathcal{M}_{2g-1-2id(P)-2jd(Q)}} \chi_{PQ}(h)\bigg).
 \end{align}
 As in the previous subsection, the terms with $d(PQ)$ odd shall lead to $V\ne\square$ after applying Lemma \ref{L2}, and their contribution, as before, is bounded by $O(N^2q^{N/2-2g})$. We are left with the terms with $d(PQ)$ even. 
From Lemmas \ref{L2} and \ref{L3}, the expression inside the bracket is equal to
\begin{align*}
&\frac{q^{2g+1}}{|P|^{2i+1/2}|Q|^{2j+1/2}}\bigg(q\sum_{V \in \mathcal{M}_{\leq d(PQ)+2id(P)+2jd(Q)-2g-3}} \chi_{PQ}(V)- \sum_{V \in \mathcal{M}_{\leq d(PQ)+2id(P)+2jd(Q)-2g-2}} \chi_{PQ}(V) \bigg)\\
 &\ -\frac{q^{2g}}{|P|^{2i+1/2}|Q|^{2j+1/2}}\bigg(q\sum_{V \in \mathcal{M}_{\leq d(PQ)+2id(P)+2jd(Q)-2g-1}} \chi_{PQ}(V)- \sum_{V \in \mathcal{M}_{\leq d(PQ)+2id(P)+2jd(Q)-2g}} \chi_{PQ}(V) \bigg).
\end{align*} 
Again the contribution from the terms $V\ne\square$ is negligible and we focus on the term with $V=\square$, which is
\begin{align}\label{500}
&\frac{q^{2g+1}}{|P|^{2i+1/2}|Q|^{2j+1/2}}\bigg(q\sum_{\substack{V \in \mathcal{M}_{\leq d(PQ)/2+id(P)+jd(Q)-g-2}\\(V,PQ)=1}}1- \sum_{\substack{V \in \mathcal{M}_{\leq d(PQ)/2+id(P)+jd(Q)-g-1}\\(V,PQ)=1}} 1\bigg)\\
 &\qquad-\frac{q^{2g}}{|P|^{2i+1/2}|Q|^{2j+1/2}}\bigg(q\sum_{\substack{V \in \mathcal{M}_{\leq d(PQ)/2+id(P)+jd(Q)-g-1}\\(V,PQ)=1}}1- \sum_{\substack{V \in \mathcal{M}_{\leq d(PQ)/2+id(P)+jd(Q)-g}\\(V,PQ)=1}} 1 \bigg).\nonumber
\end{align} 

First consider $I_{2,>}^{\textrm{oo}}(N;\alpha,\beta)$. The treatment for $I_{2,<}^{\textrm{oo}}(N;\alpha,\beta)$ is similar. From \eqref{condition} we have $id(P)+jd(Q)\leq g$, so
\[
d(V)\leq d(PQ)/2+id(P)+jd(Q)-g\leq d(PQ)/2<d(P),
\] and hence $(V,P)=1$ automatically. Note that
\begin{align*}
&q\sum_{\substack{V \in \mathcal{M}_{\leq d(PQ)/2+id(P)+jd(Q)-g-1}\\(V,Q)=1}}1- \sum_{\substack{V \in \mathcal{M}_{\leq d(PQ)/2+id(P)+jd(Q)-g}\\(V,Q)=1}} 1\\
&\qquad=\bigg(q\sum_{V \in \mathcal{M}_{\leq d(PQ)/2+id(P)+jd(Q)-g-1}}1- \sum_{V \in \mathcal{M}_{\leq d(PQ)/2+id(P)+jd(Q)-g}} 1\bigg)\\
&\qquad\qquad\qquad-\bigg(q\sum_{V \in \mathcal{M}_{\leq (d(P)-d(Q))/2+id(P)+jd(Q)-g-1}}1- \sum_{V \in \mathcal{M}_{\leq (d(P)-d(Q))/2+id(P)+jd(Q)-g}} 1\bigg)\\
&\qquad=\begin{cases}
-1 & \quad\textrm{if } (2i+1)d(P)+(2j-1)d(Q)<2g\leq (2i+1)d(P)+(2j+1)d(Q),\\
0 & \quad\textrm{otherwise}.
\end{cases}
\end{align*}
So
\begin{align*}
 \eqref{500}=\begin{cases}
-\frac{q^{2g}(q-1)}{|P|^{2i+1/2}|Q|^{2j+1/2}} & \textrm{if } (2i+1)d(P)+(2j-1)d(Q)<2g< (2i+1)d(P)+(2j+1)d(Q),\\
-\frac{q^{2g+1}}{|P|^{2i+1/2}|Q|^{2j+1/2}} & \textrm{if } (2i+1)d(P)+(2j-1)d(Q)=2g,\\
\frac{q^{2g}}{|P|^{2i+1/2}|Q|^{2j+1/2}} & \textrm{if } (2i+1)d(P)+(2j+1)d(Q)=2g,\\
0 & \textrm{otherwise.}
\end{cases}
\end{align*}
Hence $I_{2,>}^{\textrm{oo}}(N;\alpha,\beta)$ is equal to, up to an error of size  $O(N^2q^{N/2-2g})$,
\begin{align*}
&-\sum_{\substack{d(P)>d(Q)\\d(P^{2k+1}Q^{2l+1})\, \textrm{even}\,\leq N}}\,\sum_{\substack{(2i+1)d(P)+(2j-1)d(Q)< 2g\\2g< (2i+1)d(P)+(2j+1)d(Q)}}\frac{d(P)d(Q)}{|P|^{(2k+1)(1/2+\alpha)+2i+1/2}|Q|^{(2l+1)(1/2+\beta)+2j+1/2}}\\
&\quad-\frac{q}{q-1}\sum_{\substack{d(P)>d(Q)\\d(P^{2k+1}Q^{2l+1})\, \textrm{even}\,\leq N}}\,\sum_{(2i+1)d(P)+(2j-1)d(Q)=2g}\frac{d(P)d(Q)}{|P|^{(2k+1)(1/2+\alpha)+2i+1/2}|Q|^{(2l+1)(1/2+\beta)+2j+1/2}}\\
&\quad+\frac{1}{q-1}\sum_{\substack{d(P)>d(Q)\\d(P^{2k+1}Q^{2l+1})\, \textrm{even}\,\leq N}}\,\sum_{(2i+1)d(P)+(2j+1)d(Q)=2g}\frac{d(P)d(Q)}{|P|^{(2k+1)(1/2+\alpha)+2i+1/2}|Q|^{(2l+1)(1/2+\beta)+2j+1/2}}.
\end{align*}

By the Prime Polynomial Theorem, the contribution of the terms with $d(P)=m$ and $d(Q)=n$ to $I_{2,>}^{\textrm{oo}}(N;\alpha,\beta)$ for each $i,j,k,l$ is bounded by
\begin{equation}\label{bound100}
\ll q^{-(k+2i)m-(l+2j)n}.
\end{equation}
Note that $(2i+1)m+(2j+1)n\geq2g$, so this is, in particular, bounded by $O( q^{-2g+(1-k)m+(1-l)n})$. It follows that the contribution of the terms with $m+n\leq g$ is $O(q^{-g})$. For those with $m+n>g$, the condition $m>n$ leads to $m>g/2$, and it follows from \eqref{bound100} that the contribution of such terms with $i+k\geq 1$ is $O(Nq^{-g/2})$. Hence we can restrict to the case $i=k=0$ and get
\begin{align*}
I_{2,>}^{\textrm{oo}}(N;\alpha,\beta)&=-\sum_{\substack{d(P)>d(Q)\\d(PQ^{2l+1})\, \textrm{even}\,\leq N}}\,\sum_{\substack{d(P)+(2j-1)d(Q)< 2g\\2g< d(P)+(2j+1)d(Q)}}\frac{d(P)d(Q)}{|P|^{1+\alpha}|Q|^{(2l+1)(1/2+\beta)+2j+1/2}}\\
&\qquad\qquad-\frac{q}{q-1}\sum_{\substack{d(P)>d(Q)\\d(PQ^{2l+1})\leq N}}\,\sum_{d(P)+(2j-1)d(Q)=2g}\frac{d(P)d(Q)}{|P|^{1+\alpha}|Q|^{(2l+1)(1/2+\beta)+2j+1/2}}\\
&\qquad\qquad+\frac{1}{q-1}\sum_{\substack{d(P)>d(Q)\\d(PQ^{2l+1})\leq N}}\,\sum_{d(P)+(2j+1)d(Q)=2g}\frac{d(P)d(Q)}{|P|^{1+\alpha}|Q|^{(2l+1)(1/2+\beta)+2j+1/2}}\\
&\qquad\qquad+O(N^2q^{N/2-2g})+O(Nq^{-g/2}).
\end{align*}
\kommentar{\acom{In the last two terms we need $d(PQ^{2l+1})$ to be even.}\Hung{$d(P)+(2j\pm1)d(Q)=2g$, and hence $d(PQ^{2l+1})$ even.}}

We shall write
\[
I_{2,>}^{\textrm{oo}}(N;\alpha,\beta)=I_{2,>}^{\textrm{oo}\flat}(N;\alpha,\beta)+I_{2,>}^{\textrm{oo}\dagger}(N;\alpha,\beta)+O(N^2q^{N/2-2g})+O(Nq^{-g/2})
\]
to separate the cases $j+l\geq 1$ and $j=l=0$, respectively. For $I_{2,>}^{\textrm{oo}\flat}(\alpha,\beta)$, by the Perron formula we have
\[
I_{2,>}^{\textrm{oo}\flat}(N;\alpha,\beta)=\frac{1}{2\pi i}\oint_{|u|=r}J_{2,>}^{\textrm{oo}\flat}(u,\alpha,\beta)\,\frac{du}{u^{N+1}(1-u)}
\]
for any $r<q^{-\varepsilon}$, where
\begin{align*}
J_{2,>}^{\textrm{oo}\flat}(u,\alpha,\beta)=&-\sum_{l+j\geq1}\,\sum_{\substack{d(P)>d(Q)\\d(PQ)\, \textrm{even}\\d(P)+(2j-1)d(Q)< 2g\\2g< d(P)+(2j+1)d(Q)}}\frac{d(P)d(Q)u^{d(P)+(2l+1)d(Q)}}{|P|^{1+\alpha}|Q|^{(2l+1)(1/2+\beta)+2j+1/2}}\\
&\qquad\qquad-\frac{q}{q-1}\,\sum_{l+j\geq1}\,\sum_{\substack{d(P)>d(Q)\\d(P)+(2j-1)d(Q)=2g}}\frac{d(P)d(Q)u^{d(P)+(2l+1)d(Q)}}{|P|^{1+\alpha}|Q|^{(2l+1)(1/2+\beta)+2j+1/2}}\\
&\qquad\qquad+\frac{1}{q-1}\,\sum_{l+j\geq1}\,\sum_{\substack{d(P)>d(Q)\\d(P)+(2j+1)d(Q)=2g}}\frac{d(P)d(Q)u^{d(P)+(2l+1)d(Q)}}{|P|^{1+\alpha}|Q|^{(2l+1)(1/2+\beta)+2j+1/2}}.
\end{align*}
Given $Q$, from the Prime Polynomial Theorem we have
\begin{align*}
&\sum_{\substack{dP)>d(Q)\\d(PQ)\, \textrm{even}\\d(P)+(2j-1)d(Q)< 2g\\2g< d(P)+(2j+1)d(Q)}}\frac{d(P)u^{d(PQ)}}{|P|^{1+\alpha}}=\sum_{\max\{d(Q),g-jd(Q)\}<n<g-(j-1)d(Q)}\frac{|Q|^\alpha u^{2n}}{q^{2n\alpha}}\big(1+O(q^{-n}|Q|^{1/2})\big)\\
&=\begin{cases}-q^{-2g\alpha}u^{2g-2jd(Q)}|Q|^{(2j+1)\alpha}\frac{u^{2}}{u^2-q^{2\alpha}}+q^{-2g\alpha}u^{2g-2(j-1)d(Q)}|Q|^{(2j-1)\alpha}\frac{q^{2\alpha}}{u^2-q^{2\alpha}}+O(q^{-g}|Q|^{j+1/2})\\
\qquad\qquad\qquad\qquad\qquad\qquad\qquad\qquad\qquad\qquad\qquad\qquad\text{if }(j+1)d(Q)<g,\\
-u^{2d(Q)}|Q|^{-\alpha}\frac{u^{2}}{u^2-q^{2\alpha}}+q^{-2g\alpha}u^{2g-2(j-1)d(Q)}|Q|^{(2j-1)\alpha}\frac{q^{2\alpha}}{u^2-q^{2\alpha}}+O(|Q|^{-1/2})\\
\qquad\qquad\qquad\qquad\qquad\qquad\qquad\qquad\qquad\qquad\qquad\qquad\text{if }jd(Q)< g\leq (j+1)d(Q).
\end{cases}
\end{align*}
\kommentar{\acom{For the error term I get that we have an extra $|Q|^{2rj}$, and since $r = \alpha-\log_q u$ if $\alpha$ is for example close to $1/4$ then the error term is much bigger.}\Hung{Again, we shall assume $\alpha \ll 1/g$, and I think the above is fine}}
Hence
\begin{align*}
&J_{2,>}^{\textrm{oo}\flat}(u,\alpha,\beta)=\,q^{-2g\alpha}u^{2g}\bigg(\frac{u^{2}}{u^2-q^{2\alpha}}+\frac{1}{q-1}\bigg)\sum_{l+j\geq1}\,\sum_{(j+1)d(Q)<g}\frac{d(Q)u^{2(l-j)d(Q)}}{|Q|^{1-\alpha+\beta+l(1+2\beta)+2j(1-\alpha)}}\\
&\qquad+\frac{u^{2}}{u^2-q^{2\alpha}}\sum_{l+j\geq1}\,\sum_{jd(Q)< g\leq (j+1)d(Q)}\frac{d(Q)u^{2(l+1)d(Q)}}{|Q|^{1+\alpha+\beta+l(1+2\beta)+2j}}\\
&\qquad-q^{-2g\alpha}u^{2g}\bigg(\frac{q^{2\alpha}}{u^2-q^{2\alpha}}+\frac{q}{q-1}\bigg)\sum_{l+j\geq1}\,\sum_{jd(Q)<g}\frac{d(Q)u^{2(l-j+1)d(Q)}}{|Q|^{1+\alpha+\beta+l(1+2\beta)+2j(1-\alpha)}}+O(q^{-g/2}).
\end{align*}

By the Prime Polynomial Theorem again, it is easy to see that the second expression is bounded by $O(q^{-g})$. Also, we can extend the sum over $Q$ in the third expression to all of $Q\in\mathcal{P}$ at the cost of an error of size $O_\varepsilon(q^{-2g+\varepsilon g})$. For the first expression, we write
\begin{align*}
&\sum_{l+j\geq1}\,\sum_{(j+1)d(Q)<g}\frac{d(Q)}{|Q|^{1-\alpha+\beta}}\,x^ly^j = \sum_{d(Q)<g}\frac{d(Q)}{|Q|^{1-\alpha+\beta}}\sum_{\substack{l+j\geq1\\j< g/d(Q)-1}}x^ly^j\\
&\qquad\qquad=\sum_{d(Q)<g}\frac{d(Q)}{|Q|^{1-\alpha+\beta}}\sum_{l+j\geq1}x^ly^j+O_\varepsilon(q^{-g+\varepsilon g})\\
&\qquad\qquad=\sum_{d(Q)<g}\frac{d(Q)}{|Q|^{1-\alpha+\beta}}\,\bigg(\frac{x}{1-x}+\frac{y}{(1-x)(1-y)}\bigg)+O_\varepsilon(q^{-g+\varepsilon g}).
\end{align*}
\kommentar{\acom{I'm not sure that this is correct. For example if $j=1$ and $l=0$ I get $q^{-2g/3} q^{(\alpha-\beta)g/3}$ and again if $\alpha \sim 1/4$ the above can be bigger. Maybe you're assuming everywhere that $\Re(\alpha), \Re(\beta) \ll 1/g$?}\Hung{Same comment as above}}
The identities in \eqref{A22} and an argument similar to that used in the evaluation of $J_{2}^{\textrm{oe}}(u,\alpha,\beta)$ in equation \eqref{factno1} then imply that
\begin{align}\label{factno2}
&J_{2,>}^{\textrm{oo}\flat}(u,\alpha,\beta)=\,q^{-2g\alpha}u^{2g}\mathcal{A}_2(u,\alpha)\bigg(\sum_{d(Q)<g}\frac{d(Q)u^{2d(Q)}}{|Q|^{1-\alpha+\beta}(|Q|^{1+2\beta}-u^{2d(Q)})}\nonumber\\
&\qquad+\sum_{d(Q)<g}\frac{d(Q)|Q|^{\alpha+\beta}}{(|Q|^{2(1-\alpha)}u^{2d(Q)}-1)(|Q|^{1+2\beta}-u^{2d(Q)})}\\
&\qquad-\sum_{d(Q)<g}\frac{d(Q)|Q|^{2-3\alpha+\beta}u^{4d(Q)}}{(|Q|^{2(1-\alpha)}u^{2d(Q)}-1)(|Q|^{1+2\beta}-u^{2d(Q)})}+\sum_{d(Q)<g}\frac{d(Q)u^{2d(Q)}}{|Q|^{1+\alpha+\beta}} \bigg)+O(q^{-g/2}).\nonumber
\end{align}

For $I_{2,>}^{\textrm{oo}\dagger}(N;\alpha,\beta)$, by the Perron formula we have
\[
I_{2,>}^{\textrm{oo}\dagger}(N;\alpha,\beta)=\frac{1}{2\pi i}\oint_{|u|=r}J_{2,>}^{\textrm{oo}\dagger}(u,\alpha,\beta)\,\frac{du}{u^{N+1}(1-u)} \label{idag}
\]
for any $r<q^{-\varepsilon}$, where 
\begin{align*}
J_{2,>}^{\textrm{oo}\dagger}(u,\alpha,\beta)&=-\sum_{\substack{d(P)>d(Q)\\d(PQ)\, \textrm{even}}}\,\sum_{\substack{d(P)-d(Q)< 2g\\2g< d(P)+d(Q)}}\frac{d(P)d(Q)u^{d(PQ)}}{|P|^{1+\alpha}|Q|^{1+\beta}}-\frac{q}{q-1}\sum_{d(P)-d(Q)=2g}\frac{d(P)d(Q)u^{d(PQ)}}{|P|^{1+\alpha}|Q|^{1+\beta}}\\
&\qquad\qquad+\frac{1}{q-1}\sum_{d(P)>d(Q)}\,\sum_{d(P)+d(Q)=2g}\frac{d(P)d(Q)u^{d(PQ)}}{|P|^{1+\alpha}|Q|^{1+\beta}}.
\end{align*}
The last two terms can be evaluated using the Prime Polynomial Theorem. Concerning the first term, note that given $Q$,
\begin{align*}
&\sum_{\substack{d(P)>d(Q)\\d(PQ)\, \textrm{even}}}\,\sum_{\substack{d(P)-d(Q)< 2g\\2g< d(P)+d(Q)}}\frac{d(P)u^{d(PQ)}}{|P|^{1+\alpha}}=\sum_{\max\{d(Q),g\}<n<g+d(Q)}\frac{|Q|^\alpha u^{2n}}{q^{2n\alpha}}\big(1+O(q^{-n}|Q|^{1/2})\big)\\
&\qquad\qquad=\begin{cases}
-q^{-2g\alpha}u^{2g}|Q|^{\alpha}\frac{u^{2}}{u^2-q^{2\alpha}}+q^{-2g\alpha}u^{2g+2d(Q)}|Q|^{-\alpha}\frac{q^{2\alpha}}{u^2-q^{2\alpha}}+O(q^{-g/2}) & \textrm{if }d(Q)< g,\\
-|Q|^{-\alpha}u^{2d(Q)}\frac{u^{2}}{u^2-q^{2\alpha}}+q^{-2g\alpha}u^{2g+2d(Q)}|Q|^{-\alpha}\frac{q^{2\alpha}}{u^2-q^{2\alpha}}+O(q^{-g/2}) & \textrm{if }d(Q)\geq g,
\end{cases}
\end{align*}
by writing $d(PQ)=2n$.
\kommentar{\acom{I agree with the answer, but I find the way it is written confusing. I'd write
\begin{align*}
\sum_{\substack{d(P)>d(Q)\\d(PQ)\, \textrm{even}}}\,\sum_{\substack{d(P)-d(Q)< 2g\\2g< d(P)+d(Q)}}\frac{d(P)}{|P|^{1+r}} = \sum_{\substack{\max\{d(Q),2g-d(Q)\}<n<2g+d(Q) \\ n \equiv d(Q) \pmod 2}} \Big(q^{-nr} + O(q^{-n/2-nr}) \Big).
\end{align*}

}\Hung{I try to keep it consistent with the way I wrote before, but either way is fine}}
So
\begin{align*}
&-\sum_{\substack{d(P)>d(Q)\\d(PQ)\, \textrm{even}}}\,\sum_{\substack{d(P)-d(Q)< 2g\\2g< d(P)+d(Q)}}\frac{d(P)d(Q)u^{d(PQ)}}{|P|^{1+\alpha}|Q|^{1+\beta}}=q^{-2g\alpha}u^{2g}\frac{u^{2}}{u^2-q^{2\alpha}}\sum_{d(Q)< g}\frac{d(Q)}{|Q|^{1-\alpha+\beta}}\\
&\qquad\qquad +\frac{u^{2}}{u^2-q^{2\alpha}}\sum_{d(Q)\geq g}\frac{d(Q)u^{2d(Q)}}{|Q|^{1+\alpha+\beta}}-q^{-2g\alpha}u^{2g}\frac{q^{2\alpha}}{u^2-q^{2\alpha}}\sum_{Q\in\mathcal{P}}\frac{d(Q)u^{2d(Q)}}{|Q|^{1+\alpha+\beta}}+O(q^{-g/2}).
\end{align*}
Hence, using \eqref{A22}, we have
\begin{align*}
J_{2,>}^{\textrm{oo}\dagger}(u,\alpha,\beta)&=q^{-2g\alpha}u^{2g}\bigg(\frac{u^{2}}{u^2-q^{2\alpha}}+\frac{1}{q-1}\bigg)\sum_{d(Q)< g}\frac{d(Q)}{|Q|^{1-\alpha+\beta}}+\frac{u^{2}}{u^2-q^{2\alpha}}\sum_{d(Q)\geq g}\frac{d(Q)u^{2d(Q)}}{|Q|^{1+\alpha+\beta}}\nonumber\\
&\qquad\qquad-q^{-2g\alpha}u^{2g}\bigg(\frac{q^{2\alpha}}{u^2-q^{2\alpha}}+\frac{q}{q-1}\bigg)\sum_{Q\in\mathcal{P}}\frac{d(Q)u^{2d(Q)}}{|Q|^{1+\alpha+\beta}}+O(q^{-g/2})\nonumber\\
&=q^{-2g\alpha}u^{2g}\mathcal{A}_2(u,\alpha)\bigg(\sum_{d(Q)< g}\frac{d(Q)}{|Q|^{1-\alpha+\beta}}-\sum_{d(Q)<g}\frac{d(Q)u^{2d(Q)}}{|Q|^{1+\alpha+\beta}}\bigg)\\
&\qquad\qquad+\frac{u^{2}}{u^2-q^{2\alpha}}\sum_{d(Q)\geq g}\frac{d(Q)u^{2d(Q)}}{|Q|^{1+\alpha+\beta}}+O(q^{-g/2}),\nonumber
\end{align*}
where in the second identity we truncated the second sum over $Q$ using a similar argument as before.
For the third term, from the Prime Polynomial Theorem we have
\[
\sum_{d(Q)\geq g}\frac{d(Q)u^{2d(Q)}}{|Q|^{1+\alpha+\beta}}=\sum_{n\geq g}\frac{u^{2n}}{q^{n(\alpha+\beta)}}\big(1+O(q^{-n/2})\big)=-q^{-g(\alpha+\beta)}u^{2g}\frac{q^{\alpha+\beta}}{u^2-q^{\alpha+\beta}}+O(q^{-g/2}).
\]
\kommentar{\acom{Again, we need $\Re(\alpha) \geq 0$ for the error term.}\Hung{I rewrite things a bit differently}}
Thus,
\begin{align}
J_{2,>}^{\textrm{oo}\dagger}(u,\alpha,\beta)&=q^{-2g\alpha}u^{2g}\mathcal{A}_2(u,\alpha)\bigg(\sum_{d(Q)< g}\frac{d(Q)}{|Q|^{1-\alpha+\beta}}-\sum_{d(Q)<g}\frac{d(Q)u^{2d(Q)}}{|Q|^{1+\alpha+\beta}}\bigg)\nonumber \\
&\qquad\qquad-q^{-g(\alpha+\beta)}u^{2g}\frac{q^{\alpha+\beta}u^2}{(u^2-q^{2\alpha})(u^2-q^{\alpha+\beta})}+O(q^{-g/2}). \nonumber \\
 \label{factno3}
\end{align}

Combining \eqref{factno2} and \eqref{factno3} we obtain
\begin{align*}
I_{2,>}^{\textrm{oo}}(N;\alpha,\beta)=\frac{1}{2\pi i}\oint_{|u|=r}J_{2,>}^{\textrm{oo}}(u,\alpha,\beta)\,\frac{du}{u^{N+1}(1-u)}+O(N^2q^{N/2-2g})+O(Nq^{-g/2}),
\end{align*}
%\acom{Shouldn't we have $O(N^2 q^{N/2-2g})$ as well?}
where
\begin{align}\label{factno4}
 J_{2,>}^{\textrm{oo}}(u,\alpha,\beta)&= q^{-2g\alpha}u^{2g}\mathcal{A}_2(u,\alpha)\bigg(\sum_{d(Q)<g}\frac{d(Q)u^{2d(Q)}}{|Q|^{1-\alpha+\beta}(|Q|^{1+2\beta}-u^{2d(Q)})}\nonumber\\
&\qquad\qquad+\sum_{d(q)<g}\frac{d(Q)(|Q|^{\alpha+\beta}-|Q|^{2-3\alpha+\beta}u^{4d(Q)})}{(|Q|^{2(1-\alpha)}u^{2d(Q)}-1)(|Q|^{1+2\beta}-1)}+ \sum_{d(Q)<g} \frac{1}{|Q|^{1-\alpha+\beta}} \bigg) \\
 &\qquad\qquad -q^{-g(\alpha+\beta)}u^{2g}\frac{q^{\alpha+\beta}u^2}{(u^2-q^{2\alpha})(u^2-q^{\alpha+\beta})} +O(q^{-g/2}).\nonumber
\end{align}
\kommentar{\begin{align}\label{factno4}
J_{2,>}^{\textrm{oo}}(u,\alpha,\beta)&=-A_2(-r,s,r,s)q^{-2gr}\zeta_q(1-2r)\sum_{Q\in\mathcal{P}}\frac{d(Q)(|Q|^{2-3r+s}-|Q|^{2-r+s})}{(|Q|^{2(1-r)}-1)(|Q|^{1+2s}-1)}\\
&\qquad\qquad-\frac{q^{-g(r+s)}}{1-q^{r-s}}\bigg(\frac{1}{1-q^{r+s}}+\frac{1}{q-1}\bigg).\nonumber
\end{align}
}

Now consider $I_{2,=}^{\textrm{oo}}(N;\alpha,\beta)$. As before we will have $(V,PQ)=1$ automatically in \eqref{500}. So
\begin{align*}
\eqref{500}&=\frac{q^{2g+1}}{|P|^{2i+2j+1}}\bigg(q\sum_{V \in \mathcal{M}_{\leq (i+j+1)d(P)-g-2}}1- \sum_{V \in \mathcal{M}_{\leq (i+j+1)d(P)-g-1}} 1\bigg)\\
&\qquad\qquad-\frac{q^{2g}}{|P|^{2i+2j+1}}\bigg(q\sum_{V \in \mathcal{M}_{\leq (i+j+1)d(P)-g-1}}1- \sum_{V \in \mathcal{M}_{\leq (i+j+1)d(P)-g}} 1 \bigg)\\
&=\begin{cases}
-\frac{q^{2g}(q-1)}{|P|^{2i+2j+1}} & \textrm{if } (i+j+1)d(P)>g,\\
\frac{q^{2g}}{|P|^{2i+2j+1}} & \textrm{if } (i+j+1)d(P)=g,\\
0 & \textrm{otherwise.}
\end{cases}
\end{align*}
Hence
\begin{align*}
&I_{2,=}^{\textrm{oo}}(N;\alpha,\beta)=-\sum_{\substack{P\ne Q\\d(P^{2k+2l+2})\leq N\\d(P)=d(Q)>g/(i+j+1)}}\frac{d(P)^2}{|P|^{(2k+1)(1/2+\alpha)+(2l+1)(1/2+\beta)+2i+2j+1}}\\
&\qquad\qquad+\frac{1}{q-1}\sum_{\substack{P\ne Q\\d(P^{2k+2l+2})\leq N\\d(P)=d(Q)=g/(i+j+1)}}\frac{d(P)^2}{|P|^{(2k+1)(1/2+\alpha)+(2l+1)(1/2+\beta)+2i+2j+1}}+O(N^2 q^{N/2-2g}).
\end{align*}
%\acom{We should have the error term from $V \neq \square$ as well, which should be $O(N^2 q^{N/2-2g})$.}
The same argument as before shows that the contribution of the term with $i+j+k+l\geq 1$ is $O(Nq^{-g})$. For $i=j=k=l=0$, we can ignore the condition $P\ne Q$ at the cost of $O(gq^{-g})$. So using the Perron formula we obtain that
\[
I_{2,=}^{\textrm{oo}}(N;\alpha,\beta)=\frac{1}{2\pi i}\oint_{|u|=r}J_{2,=}^{\textrm{oo}}(u,\alpha,\beta)\,\frac{du}{u^{N+1}(1-u)}+O(N^2 q^{N/2-2g})+O(gq^{-g})
\]
%\acom{Plus the error from non-squares}
for any $r<q^{-\varepsilon}$, where
\begin{align*}
J_{2,=}^{\textrm{oo}}(u,\alpha,\beta)&=-\sum_{d(P)=d(Q)>g}\frac{d(P)^2u^{2d(P)}}{|P|^{2+\alpha+\beta}}+\frac{1}{q-1}\sum_{d(P)=d(Q)=g}\frac{d(P)^2u^{2d(P)}}{|P|^{2+\alpha+\beta}}.
\end{align*}

From the Prime Polynomial Theorem we get
\begin{align}\label{fact3}
J_{2,=}^{\textrm{oo}}(u,\alpha,\beta)&=-\sum_{n>g} \frac{u^{2n}}{q^{n(\alpha+\beta)}}+\frac{q^{-g(\alpha+\beta)}u^{2g}}{q-1}+O(q^{-g/2}) \nonumber\\
&=q^{-g(\alpha+\beta)}u^{2g}\bigg(\frac{u^2}{u^2-q^{\alpha+\beta}}+\frac{1}{q-1}\bigg)+O(q^{-g/2}).
\end{align}

\subsection{Combining Type-I terms}
\label{combine}
In view of \eqref{r2}, \eqref{t2}, \eqref{factno1} and \eqref{factno4} we obtain 
\begin{align*}
I_{2}^{\textrm{oe}}(N;\alpha,\beta)+ I_{2,>}^{\textrm{oo}}(N;\alpha,\beta)&=\frac{1}{2\pi i}\oint_{|u|=r}J_2(u,\alpha,\beta)\,\frac{du}{u^{N+1}(1-u)}+O(N^2q^{N/2-2g})+O(Nq^{-g/2}),
\end{align*}
where
\begin{align*}
J_2(u,\alpha,\beta)&=R_2(u,\alpha,\beta)- \frac{q^{-g(\alpha+\beta)}u^{2g}\mathcal{A}_2(u,\alpha)}{1-q^{\alpha-\beta}}-q^{-g(\alpha+\beta)}u^{2g}\frac{q^{\alpha+\beta}u^2}{(u^2-q^{2\alpha})(u^2-q^{\alpha+\beta})}+ O(q^{-g/2})\\
&=R_2(u,\alpha,\beta)- \frac{q^{-g(\alpha+\beta)}u^{2g}}{(1-q^{\alpha-\beta})(q-1)}-\frac{q^{-g(\alpha+\beta)}u^{2g}}{1-q^{\alpha-\beta}}\frac{u^2}{u^2-q^{\alpha+\beta}}+ O(q^{-g/2}).
\end{align*}
\kommentar{\acom{We get that
\begin{align*}
J_2(u,\alpha,\beta) &= q^{-2\alpha g} u^{2g} \mathcal{A}_2(u,\alpha) \Big(\sum_{d(Q)<g} \frac{d(Q)}{|Q|^{1-\alpha+\beta}} \\
&- \sum_{Q} \frac{d(Q)(|Q|^{3a+b}- |Q|^{-1+3a-b}u^{2d(Q)}-|Q|u^{4d(Q)}+|Q|^2u^{4d}+|Q|^{1+a-b}u^{4d}-|Q|^{2-a+b}u^{4d}}{(u^{2d}-|Q|^{1+2b})(-|Q|^{2a}+|Q|^2 u^{2d})} \Big) \\
& -q^{-g(\alpha+\beta)}u^{2g}\frac{q^{\alpha+\beta}u^2}{(u^2-q^{2\alpha})(u^2-q^{\alpha+\beta})} \\
&= -q^{-2\alpha g} u^{2g} \mathcal{A}_2(u,\alpha) \sum_{Q} \frac{d(Q)(|Q|^{3a+b}- |Q|^{-1+3a-b}u^{2d(Q)}-|Q|u^{4d(Q)}+|Q|^2u^{4d}+|Q|^{1+a-b}u^{4d}-|Q|^{2-a+b}u^{4d}}{(u^{2d}-|Q|^{1+2b})(-|Q|^{2a}+|Q|^2 u^{2d})} \\
&- \frac{q^{-g(\alpha+\beta)}u^{2g} \mathcal{A}_2(u,\alpha)}{1-q^{\alpha-\beta}} - \frac{q^{-2\alpha g} u^{2g} \mathcal{A}_2(u,\alpha)}{1-q^{\beta-\alpha})}-q^{-g(\alpha+\beta)}u^{2g}\frac{q^{\alpha+\beta}u^2}{(u^2-q^{2\alpha})(u^2-q^{\alpha+\beta})}
\end{align*}}}
Similarly,
\begin{align*}
I_{2}^{\textrm{eo}}(N;\alpha,\beta)+ I_{2,<}^{\textrm{oo}}(N;\alpha,\beta)&=\frac{1}{2\pi i}\oint_{|u|=r}J_3(u,\alpha,\beta)\,\frac{du}{u^{N+1}(1-u)}+O(N^2q^{N/2-2g})+O(Nq^{-g/2}),
\end{align*}
where
\[
J_3(u,\alpha,\beta)=R_3(u,\alpha,\beta)- \frac{q^{-g(\alpha+\beta)}u^{2g}}{(1-q^{-\alpha+\beta})(q-1)}-\frac{q^{-g(\alpha+\beta)}u^{2g}}{1-q^{-\alpha+\beta}}\frac{u^2}{u^2-q^{\alpha+\beta}}+ O(q^{-g/2}).
\]

Now note that 
\[
\frac{1}{1-q^{\alpha-\beta}}+\frac{1}{1-q^{-\alpha+\beta}}=1,
\]
and hence, by using \eqref{fact3},
\begin{align*}
&I_{2}^{\textrm{oe}}(N;\alpha,\beta)+I_{2}^{\textrm{eo}}(N;\alpha,\beta)+ I_{2}^{\textrm{oo}}(N;\alpha,\beta)\\
&\qquad\qquad=\frac{1}{2\pi i}\oint_{|u|=r}\big(R_2(u,\alpha,\beta)+R_3(u,\alpha,\beta)\big)\,\frac{du}{u^{N+1}(1-u)}+O(N^2q^{N/2-2g})+O(Nq^{-g/2}).
\end{align*}

%Concerning the Type-II off-diagonals,
%\begin{align*}
%I_{2}(\alpha,\beta)=&\mathds{1}_{N\geq 4g+2}\,\bigg(-\frac{q^{-2g(\alpha+\beta)}(2q-q^{2\alpha}-q^{2\beta})}{(q-1)(1-q^{2\alpha})(1-q^{2\beta})}\\
%&\qquad+\frac{q^{-2[N/2]\alpha+2g(\alpha-\beta)}(2q-1-q^{-2\alpha+2\beta})}{(q-1)(1-q^{2\alpha})(1-q^{-2\alpha+2\beta})}+\frac{q^{-2[N/2]\beta-2g(\alpha-\beta)}(2q-1-q^{2\alpha-2\beta})}{(q-1)(1-q^{2\beta})(1-q^{2\alpha-2\beta})}\\
%&\qquad\qquad+\frac{q}{q-1}\big(q^{-2g\alpha}+q^{-2g\beta}\big)\frac{q^{-g(\alpha+\beta)}-q^{-([N/2]-g)(\alpha+\beta)}}{1-q^{\alpha+\beta}}\bigg).
%\end{align*}

%\begin{remark}
%\emph{{\color{blue}The coefficients of $q^{-2[N/2]\alpha}$, $q^{-2[N/2]\beta}$ and $q^{-2[N/2](\alpha+\beta)}$ don't seem to match with \eqref{ratio2}.}}
%\end{remark}

\end{document}